\documentclass[11pt]{article}
\usepackage[a4paper]{geometry}
\usepackage{authblk}
\usepackage{graphicx}
\usepackage{subcaption}
\usepackage{amsmath}
\usepackage{amsthm}
\usepackage{amsfonts}
\usepackage{amssymb}
\usepackage{bm}
\usepackage{pifont}
\usepackage{caption}
\usepackage{arydshln}
\usepackage{siunitx}
\usepackage{tabularx}
\usepackage{booktabs}
\usepackage[numbers, sort&compress]{natbib}
\usepackage{doi}
\hypersetup{
    colorlinks=true,
    linkcolor=blue,
    anchorcolor=blue,
    citecolor=blue,
    urlcolor=blue
}

\newcommand{\al}{\alpha}
\newcommand{\ep}{\epsilon}
\newcommand{\m}{\mathbb}
\newcommand{\ma}{\mathcal}
\newcommand{\Om}{\Omega}
\newcommand{\om}{\omega}
\newcommand{\p}{\partial}
\newcommand{\ga}{\gamma}

\newcommand{\la}{\lambda}
\newcommand{\La}{\Lambda}
\newcommand{\ka}{\kappa}
\newcommand{\vphi}{\varphi}
\newcommand{\wt}{\widetilde}
\newcommand{\into}{\hookrightarrow}

\newcommand{\per}{\text{per}}
\newcommand{\ip}[1]{\left\langle #1\right\rangle}
\newcommand{\bip}[1]{\langle\!\langle #1\rangle\!\rangle}
\newcommand{\Span}[1]{\mathrm{span}\left\{#1\right\}}
\newcommand{\df}[2]{\dfrac{d#1}{d#2}}
\newcommand{\pf}[2]{\dfrac{\partial #1}{\partial #2}}
\newcommand{\conj}[1]{\overline{#1}}
\newcommand{\cl}[1]{\operatorname{cl}\left(#1\right)}
\newcommand{\dint}{\displaystyle\int\nolimits}
\renewcommand{\binom}[2]{\begin{pmatrix}#1\\#2\end{pmatrix}}

\DeclareMathOperator{\re}{Re}
\DeclareMathOperator{\im}{Im}

\newtheorem{theorem}{Theorem}[section]

\numberwithin{equation}{section}
\allowdisplaybreaks[4]

\title{Calculating Domain of Attraction Boundary of Power Systems Based on the Gentlest Ascent Dynamics}
\author[1]{Sixu Wu\thanks{lucaswu@amss.ac.cn}}
\author[2]{Chenmin Zhang\thanks{202510183372@mail.scut.edu.cn}}
\author[3]{Aiqing Zhu\thanks{zaq@lsec.cc.ac.cn}}
\author[2]{Yang Liu\thanks{epyangliu@scut.edu.cn
}}
\author[4]{Jianxi Lin\thanks{linjianxi@csg.cn
}}
\author[1,5]{Yifa Tang\thanks{tyf@lsec.cc.ac.cn}
\footnote{corresponding author}}
\affil[1]{State Key Laboratory of Mathematical Sciences, Academy of Mathematics and Systems Science, Chinese Academy of Sciences, Beijing 100190, China}
\affil[2]{School of Electric Power Engineering, South China University of Technology, Guangzhou 510640, China}
\affil[3]{Department of Mathematics, National University of Singapore, 10 Lower Kent Ridge 4 Road, 119076, Singapore}
\affil[4]{Energy and Power Planning Laboratory of Energy Development Research Institute, China Southern Power Grid, Guangzhou 510530, China}
\affil[5]{School of Mathematical Sciences, University of Chinese Academy of Sciences, Beijing 100049, China}

\date{}

\begin{document}
\maketitle

\begin{abstract}
The power system is a fundamental public utility for human beings, and is becoming increasingly important due to the growing need for electricity all over the world. Recent large‑scale blackouts in the Iberian Peninsula and the UK raise public concern about the transient stability of power systems under impact faults / following disruptive disturbances. The transient stability of the power system is primarily determined by the post-disturbance synchronizing capability of synchronous generators at a specific operating point. This problem is formulated as identifying the domain of attraction boundary of the operating point, i.e., the asymptotically stable equilibrium, of the power system model. Taking the benchmark model for transient stability analysis of the power system dominated by synchronous generators as the research object, this report employs a gentlest ascent dynamics (GAD) method for locating 1-saddle points, an adjoint operator method for locating periodic orbits, and corresponding numerical algorithms for stable manifolds, all aimed at computing the domain of attraction boundary of generator systems. These algorithms transform the determination of the domain of attraction boundary into the construction of unstable critical elements (saddle points and periodic orbits) and their stable manifolds. Theoretically, under certain assumptions we rigorise the intuition that the domain of attraction boundary is the closure of the union of the stable manifolds of critical elements with index one, and establish a stability theory for a perturbed GAD system. Numerical experiments are conducted on a two-machine system, a three-machine system containing only saddle points, and a three-machine system containing periodic orbits, validating the effectiveness and accuracy of the proposed algorithms. The results show that the algorithms presented herein accurately capture the geometric structure of the domain of attraction boundary, providing a new numerical tool for transient stability analysis of power systems.\\     
{\bf{Keywords:}} Synchronous generator systems, Domain of attraction boundary, Gentlest ascent dynamics, Adjoint operator method, Stable manifold, Periodic orbit, Geometric algorithm
\end{abstract}
\section{Introduction}
\label{sec:introduction}
Transient stability of power systems is a fundamental concern in electrical engineering, i.e., whether a stably operating power system can reach a new steady state after experiencing an unstable transient process caused by the tripping of generators, transmission lines, or a large amount of loads in a fault event. Transient stability depends on both the initial state of the system and the post-disturbance state. A usual method is to determine, through numerical simulation, whether a system subject to a specific disturbance from a given initial state can reach a new stable state --- a time-domain simulation approach. From a different perspective, one may find all states that evolve to a given stable state of a power system, i.e., finding the analytical description for the boundary of the domain of attraction of its stable equilibrium point --- the direct method.

The concept of the domain of attraction (DOA) is well established in nonlinear systems theory \cite{DEDS}. For an asymptotically stable equilibrium point (ASEP) of a nonlinear dynamical system, the DOA is defined as the set of all initial conditions whose trajectories converge to the ASEP as time tends to infinity \cite{DEDS}. In power systems, the DOA of a post-fault stable equilibrium point determines the set of admissible post-fault states that can lead to stable recovery. Therefore, estimating or exactly characterizing the DOA boundary is of paramount importance for power system planning, operation, and emergency control.

Existing approaches for DOA estimation can be grouped into two broad categories: indirect (non-direct) and direct. Among the indirect ones, we find time-domain simulation \cite{PSSC}, backward integration (reverse trajectory) techniques \cite{Genesio1985}, and stable manifold computation \cite{Chiang1988, DMSA}. Time-domain simulation is widely adopted in industry due to its high reliability and applicability to almost all kinds of models \cite{Stott1979}. However, it is computationally expensive, especially for large-scale systems, and fails to provide a quantitative measure of stability boundary or guidance for preventive control. Reverse trajectory methods can alleviate the conservativeness of initial estimates \cite{Bakhshizadeh2024} but suffer from exponential growth of computational burden with system dimension. Stable manifold methods, rooted in the geometric theory of dynamical systems, can deliver the exact DOA boundary by constructing the stable manifolds of unstable equilibrium points (UEPs) lying on that boundary \cite{Chiang1988, DMSA}. Nevertheless, identifying all such UEPs becomes computationally prohibitive even for moderately sized systems, and constructing high-dimensional stable manifolds remains a formidable challenge \cite{Hurayb2021, Kabalan2019}.

Direct approaches, for instance energy-based methods and those relying on Lyapunov's second theorem, aim to judge stability without numerically solving the differential equations, thus circumventing expensive post-fault integration \cite{Fouad1991, Pai1989}. Among these, the energy function technique stands as the most mature. For lossless network-reduction models, analytical energy functions can be derived via the first-integral principle \cite{Chu1999, Chu2005}. For lossy systems where transfer conductances cannot be ignored, numerical energy functions based on ray or trapezoidal approximations have been proposed \cite{Athay1979, Pai1977}. However, due to unavoidable numerical errors, the time derivatives of such numerically constructed functions are not guaranteed to be non-positive, potentially leading to erroneous stability assessments \cite{Uemura1972}. Moreover, key challenges persist: systematically building valid energy functions for systems that include transfer conductances and complex components, and reliably determining the stability boundary via methods such as the nearest UEP scheme \cite{ElAbiad1966, Chiang1989}, the controlling UEP method \cite{Chiang1987, Chiang1994}, or the potential energy boundary surface (PEBS) technique \cite{Kakimoto1984, Chiang1988a}.

In recent years, Lyapunov-based methods have witnessed significant progress. In particular, sum-of-squares (SOS) programming offers a systematic route to construct Lyapunov polynomials and approximate the DOA via semi-definite programming \cite{Parrilo2000, Papachristodoulou2005}. The expanding interior algorithm \cite{Anghel2013} and its improved variants \cite{Izumi2018, Liu2024} have been successfully applied to power system stability analysis. Yet the computational burden of SOS grows dramatically with the system dimension and the polynomial degree \cite{Chesi2011}. Linear matrix inequality (LMI) techniques and Takagi-Sugeno fuzzy models \cite{Takagi1985, Guerra2004} have also been employed for DOA estimation \cite{Li2023, Tang2024}, but they encounter an exponential increase in the number of fuzzy rules as the number of nonlinearities grows \cite{Ding2024}. Neural-network-based Lyapunov constructions \cite{Chang2019, Huang2022, Zhao2022, Liu2024a} enable rapid online estimation; nevertheless, their black-box nature raises concerns about interpretability and the lack of rigorous stability guarantees near the DOA boundary \cite{Zhang2021}.

A fundamentally different route to obtaining the exact DOA boundary originates from nonlinear dynamical systems theory. In a seminal work, Chiang, Hirsch, and Wu \cite{Chiang1988} established that, under certain hyperbolicity and transversality assumptions, the boundary of the domain of attraction of an asymptotically stable equilibrium point is contained in the union of the stable manifolds of the unstable equilibrium points on the boundary. Subsequently, Fisher et al. \cite{COSR2, HCRA} identified a flaw in the original lemma and strengthened the hypotheses while preserving the conclusion. For systems that satisfy these assumptions, the DOA boundary can be constructed by locating all critical elements (equilibrium points or periodic orbits) on the boundary and then computing their stable manifolds. Because stable manifolds of higher-index critical elements have lower dimensions, they occupy measure zero in the state space. Consequently, from a practical perspective, it suffices to consider only index-one critical elements, i.e., 1-saddle points and index-one periodic orbits. This geometric viewpoint reduces the DOA boundary computation to two tasks: (1) finding all 1-saddles and index-one periodic orbits on the boundary, and (2) computing their stable manifolds.

Inspired by this geometric viewpoint, this paper aims to develop novel numerical algorithms for computing the exact DOA boundary of power systems dominated by synchronous generators, by integrating the gentlest ascent dynamics (GAD) method \cite{GAD} and its infinite-dimensional extensions. The GAD method, originally proposed by E and Zhou \cite{GAD} for finding saddle points in chemical reaction dynamics, is an elegant augmented dynamical system that converts a 1-saddle point of the original system into an asymptotically stable equilibrium point of an augmented system. We apply the GAD method for the first time to locate 1-saddle points in power system DOA computation, achieving significantly faster convergence than brute-force search methods \cite{DMSA, SRNDS}. For index-one periodic orbits, we employ an adjoint-operator-based descent method \cite{CPOH} to locate the periodic orbit itself, and then extend the GAD philosophy to function spaces: we introduce an artificial viscosity term and prove that the perturbed GAD system is asymptotically stable, thereby enabling the numerical computation of the unstable eigen-direction of the periodic orbit and, subsequently, its stable manifold.  A comprehensive exposition of geometric algorithms and machine learning methods for dynamical systems and their applications can be found in \cite{Tang2026}.

We further provide a theoretical justification that under an additional topological homeomorphism assumption (A5), the closure of the union of the stable manifolds of index-one critical elements exactly equals the DOA boundary. This result establishes the theoretical foundation for our numerical approach.

The effectiveness and accuracy of the proposed algorithms are demonstrated on three benchmark systems: a two-machine system, a three-machine system containing only 1-saddles, and a three-machine system containing index-one periodic orbits. In the last example, the stable manifolds of the periodic orbits are successfully computed, and together with the stable manifolds of the 1-saddles, they constitute the complete DOA boundary. Numerical results show that the algorithms accurately capture the geometric structure of the DOA boundary, providing a new numerical tool for transient stability analysis of power systems.

The rest of this paper is organized as follows. Section \ref{sec:SGMPS} presents the synchronous-generator-based multi-machine power system model and shows how the problem reduces to a lower-dimensional angle dynamics. Section \ref{sec:SMDOA} reviews the theoretical foundation of the DOA boundary and proves that only index-one critical elements are needed. Section \ref{sec:NADOAB} describes the numerical algorithms for 1-saddle points, periodic orbits, and their stable manifolds. Section \ref{sec:NE} reports the numerical experiments. Section \ref{sec:conclusion} concludes the paper.

\section{Synchronous Generators-based Multi-Machine Power Systems}
\label{sec:SGMPS}
Consider a multi-machine power system consists of $n+1$ synchronous generators, each abstractly simplified as a rotating mass. Thus, it can be described by $n+1$ angles and the corresponding $n+1$ angular velocities. In engineering practice, one generator is often chosen as the reference machine, and the differences in angle and angular velocity of the other generators with respect to it are of interest, thereby reducing the system to $2n$ dimensions. Under the electro-mechanical transient model, the system is governed by the following dynamics \cite{Liu2024}
    \begin{equation}\label{eq:delomg}
    \left\{
        \begin{aligned}
            \dot \delta_1&=\om_1\\
            &\dots\\
            \dot \delta_n&=\om_n\\
            \dot \om_1&=\frac{\pi f_{\mathrm{n}}}{H_1}\left(P_{\mathrm{m1}}-P_{\mathrm{e1}}(\delta)\right)-\frac{D_1}{2H_1}\om_1\\
            &\dots\\
            \dot \om_n&=\frac{\pi f_{\mathrm{n}}}{H_n}\left(P_{\mathrm{m}n}-P_{\mathrm{e}n}(\delta)\right)-\frac{D_n}{2H_n}\om_n.
        \end{aligned}
    \right.
    \end{equation}
    Here $\delta=\left(\delta_1,\dots,\delta_n\right)^\top$ and $\om=\left(\om_1,\dots,\om_n\right)^\top$ denote the relative angles and angular velocities, respectively. The constant $f_{\mathrm{n}}>0$ represents the nominal frequency of the system in $\text{Hz}$, $H_i>0$ is the inertia time constant of the $i$-th generator in seconds, $D_i\geq 0$ is the friction coefficient of the $i$-th generator, and $P_{\mathrm{m}i}$ is the mechanical power of the $i$-th generator.

    The variable $P_{\mathrm{e}i}$ is the electromagnetic output power of the $i$-th generator, given by
    \[
        P_{\mathrm{e}i}=\sum_{j=1}^{n+1}E_iE_j\big(G_{ij}\cos(\delta_i-\delta_j)+B_{ij}\sin(\delta_i-\delta_j)\big),\quad \forall 1\leq i\leq n,
    \]
    where $\delta_{n+1}=0$ is the reference machine angle, the constant $E_i>0$ is the internal voltage of the $i$-th generator, the constant $G_{ij}\in \m R$ represents the reduced network conductance between generators $i$ and $j$, and the constant $B_{ij}\in \m R$ is the reduced network susceptance between generators $i$ and $j$, with both matrices $G$ and $B$ symmetric.

    It is easy to see that one equilibrium point of system \eqref{eq:delomg} is $\delta=\om=0$; this is because the original system has been translated in variables for computational convenience. Under actual operating conditions, the parameters are chosen so that $\delta=\om=0$ becomes an asymptotically stable equilibrium point of \eqref{eq:delomg}. Moreover, $D_i/H_i$ reflects the relative relationship between the damping intensity and the inertia magnitude of the $i$-th generator rotor; when the generators operate under similar conditions, this ratio is nearly constant. Therefore, we assume there exists a constant $\eta> 0$ such that $\eta=D_i/(2H_i)$ for all $1\leq i\leq n$. Under this assumption, it has been proved \cite{Liu2024} that the $k$-saddles of system \eqref{eq:delomg} correspond exactly to the $k$-saddles of the reduced system
    \begin{equation}\label{eq:del}
        \left\{
            \begin{aligned}
                \dot \delta_1&=\frac{\pi f_{\mathrm{n}}}{H_1}\left(P_{\mathrm{m}1}-P_{\mathrm{e}1}(\delta)\right)\\
                &\dots\\
                \dot \delta_n&=\frac{\pi f_{\mathrm{n}}}{H_n}\left(P_{\mathrm{m}n}-P_{\mathrm{e}n}(\delta)\right)
            \end{aligned}
        \right.
    \end{equation}
    i.e., $(\delta_*,0)$ is a $k$-saddle of \eqref{eq:delomg} if and only if $\delta_*$ is a $k$-saddle of \eqref{eq:del}. Furthermore, it is easy to see that the eigen-directions of $(\delta_*,0)$ are completely determined by those of $\delta_*$. Indeed, let $J$ be the Jacobian matrix of \eqref{eq:del} at $\delta_*$, and let \eqref{eq:delomg} have the left eigen-direction
    \begin{equation}\label{eq:muxiw}
        \begin{pmatrix}
            0 & J^\top \\ I & -\eta
        \end{pmatrix}
        \binom{\xi}{w}=\mu\binom{\xi}{w},
    \end{equation}
    then $J^\top w=\mu(\eta+\mu)w$ and $\xi=(\eta+\mu)w$, so $w$ is a left eigen-direction of \eqref{eq:del} with respect to $\mu(\eta+\mu)$. Conversely, suppose $Jw=\la w$, choose $\mu\in \m C$ satisfying $\mu(\eta+\mu)=\la$ and set $\xi=(\eta+\mu)w$, then \eqref{eq:muxiw} also holds. Thus, if $\delta_*$ is a $1$-saddle of the reduced system, the tangent spaces of the stable manifolds of the original and reduced systems satisfy
    \[
        T_{(\delta_*,0)}W^s(\delta_*,0)=\binom{(\eta+\mu)w}{w}^\perp.
    \]
    This shows that understanding the boundary of the region of attraction for the reduced system at $\delta=0$ will greatly enhance our understanding of the boundary of the region of attraction for the original system at $\delta=\om=0$. Moreover, numerical experiments have shown \cite{Liu2024} that for small $n$, the critical clearing time of the reduced system is close to that of the original system. Based on these two reasons, this chapter focuses on the numerical computation of the boundary of the region of attraction for the reduced system \eqref{eq:del} at the asymptotically stable\footnote{Both systems are assumed to be hyperbolic.} equilibrium point $\delta=0$.
\section{Stable Manifolds and the Domain of Attraction Boundary}
\label{sec:SMDOA}
Under certain hyperbolicity and transversality conditions, the domain of attraction boundary is precisely the union of the stable manifolds of the unstable equilibrium points and unstable periodic orbits on it — a geometric phenomenon first proposed by Chiang Hsiao-Dong in his 1988 paper \cite{Chiang1988}, later corrected by Michael W. Fisher et al., who identified a flaw in the lemma and strengthened the conditions while preserving the conclusion \cite{COSR2,HCRA}. This section briefly presents the corrected version, which will serve as the fundamental assumption for the numerical examples that follow.

If $p$ is an equilibrium point, the set
    \[
        W^s(p):=\left\{x\in \m R^n:\lim_{t\to +\infty} \phi(t,x)=p\right\}
    \]
is called the stable manifold of $p$. The tangent space of the stable manifold has the following properties. Suppose $p$ is a hyperbolic $k$-saddle point. For $s_1+2s_2=n-k$,
    \begin{gather*}
        \la_1,\dots,\la_{s_1}<0,\\
        \mu_1,\conj{\mu_1},\dots,\mu_{s_2},\conj{\mu_{s_2}}\in \m C-\m R
    \end{gather*}
    are all eigenvalues of $Df(p)$ with negative real parts, with corresponding eigenvectors
    \begin{gather*}
        Df(p)u_i=\la_i u_i,\quad u_i\in \m R^n-0,\quad \forall 1\leq i\leq s_1,\\
        Df(p)v_j=\mu_j v_j,\quad v_j\in \m C^n-\m R^n,\quad  \forall 1\leq j\leq s_2.
    \end{gather*}
    For $k_1+2k_2=k$,
    \begin{gather*}
        \al_1,\dots,\al_{k_1}>0,\\
        \beta_1,\conj{\beta_1},\dots,\beta_{k_2},\conj{\beta_{k_2}}\in \m C-\m R
    \end{gather*}
    are all eigenvalues of $Df(p)$ with positive real parts, with corresponding eigenvectors
    \begin{gather*}
        Df(p)w_i=\al_i w_i,\quad w_i\in \m R^n-0,\quad \forall 1\leq i\leq k_1,\\
        Df(p) \xi_j=\beta_j \xi_j,\quad \xi_j\in \m C^n-\m R^n,\quad  \forall 1\leq j\leq k_2.
    \end{gather*}

    Then there exists an $(n-k)$-dimensional differentially embedded submanifold $S\subset \m R^n$ such that $p\in S$,
    \[
        T_pS=\Span{u_1,\dots,u_{s_1}, \re v_1, \im v_1,\dots,\re v_{s_2},\dots, \im v_{s_2}},
    \]
    and for any $t\geq 0$, $\phi(t,S)\subset S$,
    \[
        \lim_{s\to +\infty} \phi(s,q)=p,\quad \forall q\in S.
    \]

    Similarly, there exists a $k$-dimensional differentially embedded submanifold $U\subset \m R^n$ such that $p\in U$,
    \[
        T_pU=\Span{w_1,\dots,w_{k_1}, \re \xi_1, \im \xi_1,\dots,\re \xi_{k_2},\dots, \im \xi_{k_2}}.
    \]
    and for any $t\leq 0$, $\phi(t,U)\subset U$,
    \[
        \lim_{s\to -\infty} \phi(s,q)=p,\quad \forall q\in U.
    \]
    Here span denotes the real vector space spanned by the vectors.

    Michael W. Fisher et al. showed that under the following assumptions, the domain of attraction boundary of an asymptotically stable equilibrium point is formed by the stable manifolds of the critical elements on it \cite{HCRA}.
    \begin{itemize}
    \item[\textbf{A1}] There exists an open neighborhood $N$ of $\p \ma A(p)$ and a natural number $k$ such that
    \[
        \Om(f)\cap N=\left\{X_i: 1\leq i\leq k\right\},
    \]
    where $X_i$ are critical elements.
    \item[\textbf{A2}] For any $q\in \p \ma A(p)$, its forward time flow $\{\phi(t,q):t\geq 0\}$ is bounded.
    \item[\textbf{A3}] All critical elements on $\p \ma A(p)$ are hyperbolic, where critical elements refer only to saddle points and periodic orbits.
    \item[\textbf{A4}] For any two critical elements $Y_i, Y_j$ (possibly the same) on $\p \ma A(p)$, the stable manifold $W^s\left(Y_i\right)$ and the unstable manifold $W^u\left(Y_j\right)$ intersect transversely.
\end{itemize}
Transverse intersection is defined as follows. Let $M$ be a smooth manifold, $S_1,S_2\subset M$ two immersed submanifolds, and $i_1:S_1\into M,\; i_2:S_2\into M$ the inclusion maps. $S_1$ and $S_2$ are said to intersect transversely at $p\in S_1\cap S_2$ if the tangent spaces satisfy
    \[
        \left(i_1\right)_*\left(T_pS_1\right)+\left(i_2\right)_*\left(T_pS_2\right)=T_pM,
    \]
    where $\left(i_1\right)_*\left(T_pS_1\right)$ denotes the image of $T_pS_1$ in $T_pM$. If they intersect transversely at every $p\in S_1\cap S_2$, then $S_1$ and $S_2$ are said to intersect transversely.

Systems satisfying assumptions \textbf{A1}--\textbf{A4} have the following property \cite{HCRA}.
\begin{theorem}\label{thm:bdary}
    Let $M$ be a compact Riemannian manifold or Euclidean space $\m R^n$, $f$ a complete $C^1$ vector field on it, and $p$ a stable equilibrium point of $f$. If \textbf{A1}--\textbf{A4} hold, then
    \[
        \p \ma A(p)=\bigcup_{i\in I} W^s\left(Y_i\right),
    \]
    where $\left\{Y_i:i\in I\right\}$ are all critical elements on $\p \ma A(p)$.
\end{theorem}
From Theorem \ref{thm:bdary}, computing the domain of attraction boundary reduces to computing the stable manifolds of all (hyperbolic) critical elements on it. Moreover the stable manifold of a critical element with index $k$, as an immersed submanifold of $\m R^n$, has dimension $n-k$. Intuitively, similar to the relationship between a plane and a line, the highest-dimensional manifold will occupy the entire area (measure), while lower-dimensional manifolds have measure zero; therefore, it suffices to compute only the stable manifolds of critical elements with index $1$. Below we introduce a new assumption \textbf{A5} and prove that under this assumption, the above intuition holds topologically.
\begin{itemize}
    \item[\textbf{A5}] There exists a homeomorphism $h:\m R^n\to \m R^n$ such that $M:=h(\p A(p))$ is an $(n-1)$-dimensional embedded submanifold of $\m R^n$ and for every critical element $Y_i$ on $\p \ma A(p)$,
    \[
        M_i:=h\left(W^s\left(Y_i\right)\right)
    \]
    is still an immersed submanifold of $\m R^n$ of the same dimension as $W^s\left(Y_i\right)$.
\end{itemize}
To prove Theorem \ref{thm:ind1bdary}, we need some preliminaries on null sets in manifolds. A subset $A$ of $\m R^n$ is called a null set if it is Lebesgue measurable and its Lebesgue measure is zero. Equivalently, for any $\ep>0$, there exist countably many open rectangles whose union contains $A$ and the sum of their volumes is less than $\ep$. A subset $S$ of an $n$-dimensional manifold $M$ is called a null set if for every chart $(U,\vphi)$ of $M$, $\vphi(S\cap U)$ is a null set in $\m R^n$. It is easy to see that a countable union of null sets in a manifold is again a null set.
\begin{theorem}\label{thm:ind1bdary}
    With the same notation as in Theorem \ref{thm:bdary} and assuming conditions \textbf{A1}--\textbf{A5} hold, then\footnote{In this paper, $\cl{A}$ denotes the closure of $A$, and $\conj{A}$ denotes the set $\{\conj z: z\in A\}$. When using the latter, it is assumed that $A\subset \m C$.}
    \[
        \p \ma A(p)=\cl{\bigcup_{i\in I_1} W^s\left(Y_i\right)},
    \]
    where $I_1=\{i\in I: Y_i \text{ has index } 1\}$. That is, $\p \ma A(p)$ is the closure in $\m R^n$ of the union of the stable manifolds of the critical elements with index $1$ on it.
\end{theorem}
\begin{proof}
    Suppose not. Then there exists an open set $U$ in the subspace topology of $\p \ma A(p)$ induced from $\m R^n$ such that
    \[
        U\cap \bigcup_{i\in I_1} W^s\left(Y_i\right)=\emptyset,
    \]
    hence by Theorem \ref{thm:bdary},
    \[
        U=\bigcup_{i\in I-I_1} U\cap W^s\left(Y_i\right).
    \]
    Since $h$ is a homeomorphism,
    \[
        h(U)=\bigcup_{i\in I-I_1} h(U)\cap M_i.
    \]
    is an open set in $M:=\p \ma A(p)$ with the subspace topology from $\m R^n$, where $M_i:=h\left(W^s\left(Y_i\right)\right)$.

    On one hand, by assumption \textbf{A5}, $M$ is an embedded submanifold of $\m R^n$, so the subspace topology induced from $\m R^n$ coincides with the manifold topology of $M$. Therefore, there exists a chart $(V,\psi)$ such that $V\subset h(U)$. Since $\psi(V)$ has positive Lebesgue measure in $\m R^{n-1}$, $h(U)$ is not a null set.

    On the other hand, by assumption \textbf{A5}, $M_i$ is an immersed submanifold of $\m R^n$ and therefore also an immersed submanifold of $M$. When $i\notin I_1$, $\dim M_i<n-1$, so by Corollary 6.11 of \cite{Lee2013}, $M_i$ is a null set. Consequently, $h(U)\cap M_i$ is also a null set. Since $h(U)$ is the union of finitely many null sets, it is also a null set, a contradiction.
\end{proof}
\section{Numerical Algorithms for the Domain of Attraction Boundary}
\label{sec:NADOAB}
\subsection{1-Saddle Points and Their Stable Manifolds}
To solve for a 1-saddle point of the system $\dot x=f(x)$, intuitively, if we reverse the component in the direction of the eigenvector corresponding to the eigenvalue with the largest real part at each point of $\m R^n$ while keeping the components in other directions unchanged, then, since the eigen-direction corresponding to the largest real part eigenvalue is precisely the unique unstable direction of a 1-saddle point, one might expect that the 1-saddle point of $f$ becomes a stable equilibrium point of the transformed new system, having only stable directions. However, this is not feasible because if we could compute the above quantities at each point, we would have already computed all equilibrium points and their indices.

This is the basic idea of the GAD method in \cite{GAD} proposed by Weinan E et al. The key is that the direction of the eigenvector corresponding to the largest real part eigenvalue does not need to be directly computed; instead, it should automatically “flow out” over time, which requires the augmented system
\begin{equation}\label{eq:gad}
    \begin{aligned}
        \dot x&=f(x)-2\dfrac{\ip{f(x),w}}{\ip{w,v}}\cdot v\\
        \dot v&=Df(x) v-\al(x,v)\cdot v\\
        \dot w&=Df(x)^\top w-\beta(x,v,w)\cdot w,
    \end{aligned}
\end{equation}
where $x,v,w: \m R\to \m R^n$, $\ip{\cdot,\cdot}$ is the standard inner product, and
\begin{equation}\label{eq:albeta}
    \begin{aligned}
        \al(x,v)&=\ip{v,Df(x)v},\\
        \beta(x,v,w)&=2\ip{w,Df(x)v}-\al(x,v).
    \end{aligned}
\end{equation}
We have the following theorem \cite{GAD}.
\begin{theorem}\label{thm:gad}
    Assume $p$ is a 1-saddle point of the system $\dot x=f(x)$, $Df(p)$ has $n$ distinct eigenvalues $\la_1,\dots,\la_n\in \m C$ with
    \[
        \la_1>0>\re \la_i,\quad \forall i\geq 2, 
    \]
    and vectors $v_1,w_1\in \m R^n-0$ are the eigenvectors of $Df(p)$ and $Df(p)^\top$ corresponding to $\la_1$, respectively, satisfying $\ip{v_1,v_1}=\ip{v_1,w_1}=1$. Then $\left(p,v_1,w_1\right)$ is an asymptotically stable equilibrium point of system \eqref{eq:gad}.

    Conversely, assume $\left(p,v^*,w^*\right)$ is a stable equilibrium point of system \eqref{eq:gad} (hence $p,v^*,w^*$ are all real vectors) satisfying $\ip{v^*,v^*}=\ip{v^*,w^*}=1$. Also assume $Df(p)$ has $n$ distinct eigenvalues $\la_1,\dots,\la_n\in \m C$ with non-zero real parts. Then, after appropriate ordering,
    \[
        \la_1>0>\re \la_i,\quad \forall i\geq 2, 
    \]
    and $Df(p)v^*=\la_1 v^*,\; Df(p)^\top w^*=\la_1 w^*$.
\end{theorem}
Hence the tangent space of the stable manifold at $p$ is
\[
    T_pW^s(p)=\Span{\re v_i, \im v_i: i>1}=\Span{w_1}^\perp,
\]
i.e., the orthogonal complement of $w_1$ in $\m R^n$.

Based on this, the numerical algorithm used in this paper to compute 1-saddle points and their stable manifolds is described by table \ref{tab:1saddleGAD}
\begin{table}[!htbp]
    \caption{\enspace GAD algorithm for 1-saddle points and stable manifolds.}
    \label{tab:1saddleGAD}
    \centering
    \begin{tabularx}{\textwidth}{>{\centering\arraybackslash}m{1.3cm}X}
        \toprule[1.5pt] 
        \textbf{S1} & Take several points appropriately on a small sphere $\p B(x^*,\delta)$ centered at the asymptotically stable equilibrium point $x^*$. Integrate along the reversed vector field $-f$ for a long time starting from these points, and denote the endpoints as $x_1,\dots,x_m$.\\
        \midrule
        \textbf{S2} & For each $x_i$, appropriately select initial unit directions $v_i^0,\,w_i^0$ (e.g., $v_i^0=w_i^0=(x_i-x^*)/\|x_i-x^*\|$). Integrate along the vector field \eqref{eq:gad} for a long time, and denote the converged endpoints as $(p_j,v_j,w_j)$.\\
        \midrule
        \textbf{S3} & Remove duplicates from all $(p_j,v_j,w_j)$. For each $j$, uniformly sample points $\xi_{j,1},\dots,\xi_{j,N}$ on $\p B(0,\delta)\cap w_j^\perp$.\\
        \midrule
        \textbf{S4} & For each $j$ and $1\leq k\leq N$, integrate along the vector field $f$ starting from $p_j+\xi_{j,k}$ until convergence. \\
        \midrule
        \textbf{S5} & Observe whether the integral curves in \textbf{S4} enclose $x^*$. If not, repeat step \textbf{S1}. \\
        \bottomrule[1.5pt]
    \end{tabularx}
\end{table}
\subsection{Periodic Orbit Localization Algorithm}
Unlike 1-saddle points, which only require computing their position and unstable eigendirection, a periodic orbit of index one also requires computing its period $T$. Therefore, this paper adopts a two-step approach to numerically solve for periodic orbits: the first step is to compute the period and the positions of points on it, and the second step is to compute its unstable eigendirection using the GAD method similar to that of the previous subsection. This subsection introduces the adjoint operator method \cite{CPOH} used for the first step, which can transform a periodic orbit of any index into a local minimum point of a numerical algorithm without requiring foreseeing the period.

A smooth map $x:\m R\to \m R^n$ is a $T$-periodic orbit if
\[
    f(x)-\om\df{x}{\tau}=0, 
\]
where $\om:=1/T$ is the frequency and $\tau:=\om t$ is the scaled time parameter. Denote the extended variable $X:=(x;\om)$. The goal of the numerical algorithm is to minimize the residual
\[
    J[X]:=\ip{r[X],r[X]}=\int_0^1 \left(f(x)-\om\df{x}{\tau}\right)^2\,d\tau,
\]
where
\[
    r[X]=r\left(X,\df{X}{\tau}\right)=\binom{f(x)-\om\df{x}{\tau}}{0}.
\]
Introducing an artificial evolution time $s$, i.e., $x=x(s,\tau),\;\om=\om(s)$, and assuming the evolution equation for $x$ is
\[
    \pf{X}{s}(s,\tau)=g[X](s,\tau),
\]
then
\[
    \df{J[X]}{s}=2\ip{\pf{}{s}r\left(X,\pf{X}{\tau}\right),r}=2\ip{A(X,g),r},
\]
where the directional derivative $A(X,g)$ is
\[
    A(X,Y):=\left.\df{}{\ep}\right|_{\ep=0}r[X+\ep g]=\binom{Df(x)y_1-y_2\pf{x}{\tau}-\om\pf{y_1}{\tau}}{0},
\]
for any $Y=(y_1;y_2), y_1\in \m R^n, y_2\in \m R$. Note that $A(X,Y)$ is a linear operator in $Y$, so there exists an adjoint operator $A^*(X,\cdot)$ satisfying
\[
    \ip{A(X,Y),Z}=\ip{Y,A^*(X,Z)},\quad\forall Z=\binom{z_1}{z_2}\in \m R^n\times \m R.
\]
In fact,
\begin{align*}
   \ip{A(X,Y),Z}
    &=\int_0^1 \left(Df(x)y_1-y_2\pf{x}{\tau}-\om\pf{y_1}{\tau}\right)\cdot z_1\,d\tau\\
    &=\int_0^1 \left(Df(x)^\top z_1+\om \pf{z_1}{\tau} \right)\cdot y_1\,d\tau\\
    &\quad -y_2\int_0^1 \pf{x}{\tau}\cdot z_1\,d\tau
\end{align*}
thus
\[
    A^*(X,Z)=\binom{Df(x)^\top z_1+\om \pf{z_1}{\tau}}{-\dint_0^1 \pf{x}{\tau}\cdot z_1\,d\tau}.
\]
Then we have
\[
    \df{J[X]}{\tau}=2\ip{A(X,g),r}=2\ip{g,A^*(X,r)}.
\]
Choose
\[
    g[X]:=-A^*(X,r)=\begin{pmatrix}-Df(x)^\top f(x)+\om\left(Df(x)^\top-Df(x)\right)\pf{x}{\tau}+\om^2\dfrac{\p^2 x}{\p \tau^2}\\
        \displaystyle\oint f(x)\cdot dx-\om\int_0^1\left(\pf{x}{\tau}\right)^2\,d\tau\end{pmatrix},
\]
where the loop integral is taken around $x$ itself, then $dJ/d\tau=-2\|g\|^2\leq 0$.

Taking the normalized time $\tau=t/T$ and a truncated Fourier expansion
\[
    x(s,\tau)=\frac{a_0(s)}{2}+\sum_{k=1}^{N}\left(a_k(s)\cos(2\pi k\tau)+b_k(s)\sin(2\pi k\tau)\right).
\]
Fix $m\geq 1$, denote $\tau_j=j/m,\, x_j=x(s,\tau_j),\,f_j=f(x_j)$ and
\[
    Df_j=Df(x_j),\quad P_j = -Df_j^\top f_j + \omega (Df_j^\top-Df_j)\,\p_\tau x_j,
\]
then the coefficients and frequency satisfy the equations
\begin{equation}\label{eq:AdjointPeriod}
    \begin{aligned}
        \frac{d a_0}{ds} &= \frac{2}{m}\sum_{j=0}^{m-1} P_j,\\
        \frac{d a_k}{ds} &= \frac{2}{m}\sum_{j=0}^{m-1} P_j\cos(2\pi k\tau_j) \;-\; \omega^2 (2\pi k)^2 a_k,\\
        \frac{d b_k}{ds} &= \frac{2}{m}\sum_{j=0}^{m-1} P_j\sin(2\pi k\tau_j) \;-\; \omega^2 (2\pi k)^2 b_k,\\
        \df{\om}{s}&=\frac{1}{m}\sum_{j=0}^{m-1} f_j\cdot \p_\tau x_j-\frac{\om}{2}\sum_{k=1}^N (2\pi k)^2(a_k^2+b_k^2).
    \end{aligned}
\end{equation}
Based on this, the numerical algorithm for locating periodic orbits used in this paper is:
\begin{table}[!htbp]
    \caption{\enspace Periodic orbit localization algorithm.}
    \label{tab:AdjointPeriod}
    \centering
    \begin{tabularx}{\textwidth}{>{\centering\arraybackslash}m{1.3cm}X}
        \toprule[1.5pt] 
        \textbf{S1} & Guess the initial frequency $\om_0$ and initial shape coefficients $a_k^0,b_k^0$ of the periodic orbit.\\
        \midrule
        \textbf{S2} & Starting from the initial values $a_k^0,b_k^0,\om_0$, integrate along the vector field \eqref{eq:AdjointPeriod} for a long time and observe convergence. If not, repeat step \textbf{S1}.\\
        \bottomrule[1.5pt]
    \end{tabularx}
\end{table}
\subsection{Numerical Algorithm for the Stable Manifold of a Periodic Orbit with One Unstable Direction}
Using the algorithm from the previous subsection, assume we have already located the $T$-periodic orbit $\ga$ with index $1$. This subsection uses the GAD method to compute its stable manifold. To reduce the dimension of the subsequent sampling algorithm, we introduce a simplified version of the GAD method \cite{SGAD}. Consider the two augmented systems derived from $\dot x=f(x)$:
\begin{subequations}
    \begin{align}
        \dot x&=f(x)-2\dfrac{\ip{f(x),v}}{\ip{v,v}}\cdot v\\
        \dot v&=Df(x) v-\ip{v,Df(x)v}\cdot v \label{eq:gadxv}
    \end{align}
\end{subequations}
and
\begin{subequations}
    \begin{align}
        \dot x&=f(x)-2\dfrac{\ip{f(x),w}}{\ip{w,w}}\cdot w\\
        \dot w&=Df(x)^\top w-\ip{w,Df(x)^\top w}\cdot w \label{eq:gadxw}
    \end{align}
\end{subequations}
Similar to the previous section, under the assumptions of Theorem \ref{thm:gad}, it is easy to prove that the point $p$ is a 1-saddle point of $f$ and $v_1,w_1$ are the left and right eigenvectors of the Jacobian matrix $Df(p)$ corresponding to the eigenvalue $\la_1>0$ if and only if $\left(p,v_1\right)$ and $\left(p,w_1\right)$ are asymptotically stable equilibrium points of \eqref{eq:gadxv} and \eqref{eq:gadxw}, respectively.

Now we attempt to apply the simplified GAD to compute the stable manifold of a periodic orbit. The periodic orbit is an equilibrium point of the evolution equation
\[
    \pf{x(s,t)}{s}=f\left(x(s,t)\right)-\pf{x(s,t)}{t},\quad x(s,t)=x(s,t+T)
\]
where $s$ is the artificial evolution time. Assuming the periodic orbit $\ga$ has exactly one unstable eigendirection, then formally, the linearized operator $L_0$ (corresponding to the Jacobian matrix $Df(x)$ in \eqref{eq:gadxv}) and its adjoint $L_0^*$ (corresponding to $Df(x)^\top$ in \eqref{eq:gadxw}) at the periodic orbit $\ga$ are respectively
\begin{align*}
    L_0v:&=\left.\df{}{\ep}\right|_{\ep=0} f(\ga+\ep v)-\df{f(\ga+\ep v)}{t}\\
    &=Df(\ga)v-\df{v}{t}.
\end{align*}
and
\[
    L_0^*w=Df(\ga)^\top w+\df{w}{t}.
\]
The latter is obtained by noting the definition of the adjoint operator
\begin{align*}
    \int_0^T v\cdot L_0^* w\,dt&=\int_0^T L_0v \cdot w\,dt\\
    &=\int_0^T \left(Df(\ga)v-\df{v}{t}\right)\cdot w\,dt\\
    &=\int_0^T v\cdot \left(Df(\ga)^\top w+\df{w}{t}\right)\,dt,
\end{align*}
where $v,w$ are $T$-periodic functions.

Just as analyzing the stability of equilibrium points of real-coefficient ordinary differential equations requires considering the complex eigenvalues of the Jacobian matrix, we first need to complexify the system and the operators. Let the Hilbert space be
\[
    H=L^2\left((0,T),\m C^n\right),
\]
where $\m T=\m R/(T\m Z)$ is the circle of circumference $T$, equipped with the standard inner product $\ip{\cdot,\cdot}$ and the complex bilinear form $\bip{\cdot,\cdot}$
\begin{align*}
    \ip{u,v}&=\int_0^T\ip{u(t),v(t)}\,dt=\int_0^T u(t)\cdot \conj{v(t)}\,dt,\\
    \bip{u,v}&=\int_0^T \bip{u(t),v(t)}\,dt=\int_0^T u(t)\cdot v(t)\,dt.
\end{align*}
The domains of the linear operators $L_0,\; L_0^*$ are the periodic Sobolev spaces
\begin{align*}
    D\left(L_0\right)=D\left(L_0^*\right)&=H^1_{\per}((0,T),\m C^n)\\&=\{v\in H:\text{weak derivative }v'\in H,\; v(0)=v(T)\}.
\end{align*}
For any $v,w\in D\left(L_0\right)$, we have $\ip{v,w'}=-\ip{v',w}$ and $\bip{v,w'}=-\bip{v',w}$, hence $\bip{v,v'}=0$.

It is easy to verify that after complexification, the expression for $L_0^*$ remains unchanged, while \eqref{eq:gadxv} and \eqref{eq:gadxw} correspond respectively to
\begin{equation}\label{eq:gadv}
    \begin{aligned}
        \pf{v(s,t)}{s}=F[v]:&=L_0v-\bip{v,L_0v}\cdot v\\
        &=L_0v-\bip{v,Df(\ga)v}\cdot v
    \end{aligned}
\end{equation}
and
\begin{equation}\label{eq:gadw}
    \begin{aligned}
        \pf{w(s,t)}{s}=F^*[w]:&=L_0^*w-\bip{w,L_0^*w}\cdot w\\
        &=L_0^*w-\bip{w,Df(\ga)w}\cdot w.
    \end{aligned}
\end{equation}
Since the periodic orbit $\ga$ itself has already been computed, we only need to compute $v$ or $w$. This section will show that the stable equilibrium point $v$ of \eqref{eq:gadv} is precisely the unstable eigendirection of $\ga$. Similarly, the stable equilibrium point $w$ of \eqref{eq:gadw} is the adjoint unstable eigendirection of $\ga$.

To facilitate energy estimates, we introduce a new Hilbert inner product on the whole $H$ that is equivalent to the original one, such that $\{\eta_\al\}$ becomes an orthonormal basis. Specifically, $\{\eta_\al\}$ forms a Riesz basis for $H$, so there exists a linear homeomorphism $U: H\to H$ mapping some orthonormal basis $\{e_\al\}$ to $\{\eta_\al\}$. Define
\[
\ip{ u, v }_* := \ip{ U^{-1}u, U^{-1}v },\quad \forall u,v\in H.
\]
Then $\ip{\cdot,\cdot}_*$ is an inner product on $H$ and is equivalent to the original inner product. Under this new inner product,
\[
\ip{ \eta_\al, \eta_\beta }_* = \ip{ e_\al, e_\beta } = \delta_{\al\beta},
\]
i.e., $\{\eta_\al\}$ is an orthonormal basis. Consequently, $X$ and $Y$, as closed subspaces spanned by certain $\eta_\al$, are orthogonal to each other. Let $u=x+y$ with $x\in X, y\in Y$; then $\ip{u,u}_*=\ip{x,x}_*+\ip{y,y}_*\geq \ip{x,x}_*$.

For any $x=\sum_{\al\in \La_0} b_\al \eta_\al\in X\cap D(L_j)$, we have $L_j x = \sum_{\al\in \La_0} b_\al \ka_\al \eta_\al$. Hence
\[
\ip{ x, L_j x}_* = \sum_{\al\in \La_0} \conj{\ka_\al}\cdot |b_\al|^2 ,
\]
and taking the real part yields
\begin{equation}\label{eq:xLjx}
    \re\ip{ x, L_j x}_* = \la_0 \sum_{\al\in \La_0} |b_\al|^2 = \la_0 \|x\|_*^2. 
\end{equation}
Similarly, for any $y = \sum_{\al\in \La_1} c_\al \eta_\al\in Y\cap D(L_j)$, we have $L_j y = \sum_{\al\in \La_1} c_\al \ka_\al \eta_\al$. Thus
\[
\ip{y, L_j y}_* = \sum_{\al\in \La_1} \conj{\ka_\al}\cdot |c_\al|^2,
\]
and taking the real part gives
\begin{equation}\label{eq:yLjy}
    \re\ip{ y, L_j y}_* = \sum_{\al\in \La_1} \re\ka_\al |c_\al|^2 \leq (\la_0-\delta)\cdot \|y\|_*^2. 
\end{equation}
Take the perturbation $u(s) = v(s)-v_j$, where $v\in D(L_0)$ satisfies \eqref{eq:gadv}. Then $u$ satisfies
\begin{equation}\label{eq:uF}
    \frac{du}{ds} = L_j u + R[u],
\end{equation}
where the remainder $R[u]$ is given by
\begin{align*}
    R[u]&= F[v_j+u] - L_j u \\
    &=-\bip{v_j,L_0u}u-\la_j\bip{u,v_j}u\\
    &\quad -\bip{u,Df(\ga) u}v_j-\bip{u,Df(\ga) u}u\\
    &=-\bip{Df(\ga)^\top v_j + v_j',u}u-\la_j\bip{u,v_j}u\\
    &\quad -\bip{u,Df(\ga)u}v_j-\bip{u,Df(\ga)u}u.
\end{align*}
Let $h=Df(\ga)^\top v_j + v_j'$ and $C=\max_{t\in[0,T]} \|Df(\ga(t))\|_H$. Then we have the estimate
\[
    \|R[u]\|_H \le \bigl( \|h\|_H + |\la_j| + C \bigr) \|u\|_H^2 + C\|u\|_H^3.
\]
Thus there exists a constant $K>0$ such that when $\|u\|_H \leq 1$,
\[
    \|R[u]\|_H \le K \|u\|_H^2.
\]
Since $\|\cdot\|_*$ is equivalent to $\|\cdot\|_H$, there exists a constant $c>0$ such that
\begin{equation}\label{eq:c}
    c\|u\|_H \le \|u\|_* \le c^{-1}\|u\|_H. 
\end{equation}
Therefore, when $\|u\|_*\leq c$, we have $\|R[u]\|_* \le c^{-3}K \|u\|_*^2$, which we denote simply as $\|R[u]\|_* \le K \|u\|_*^2$.

Let $x = P_X u$, $y = (I-P_X)u$. Applying $P_X$ to both sides of \eqref{eq:uF} yields
\[
    \left\{
        \begin{aligned}
        \dfrac{dx}{ds} &= L_j x + P_X R[x+y], \\
        \dfrac{dy}{ds} &= L_j y + (I-P_X) R[x+y].
        \end{aligned}
    \right. 
\]
Clearly $\|P_X\|_*=\|I-P_X\|_*=1$, so when $\|x+y\|_*\leq c$,
\begin{equation}\label{eq:Rxy}
    \begin{aligned}
        \|P_X R[x+y]\|_* &\leq K(\|x\|_*^2+\|y\|_*^2),\\
        \|(I-P_X)R[x+y]\|_* &\leq K(\|x\|_*^2+\|y\|_*^2).
    \end{aligned}
\end{equation}
\begin{theorem}\label{thm:vjnonstab}
    The stable real unit eigendirection $v_j$ of the periodic orbit $\ga$ is an unstable equilibrium point of the GAD system \eqref{eq:gadv}, and the adjoint stable real unit eigendirection $w_j$ of $\ga$ is an unstable equilibrium point of the GAD system \eqref{eq:gadw}.
\end{theorem}
\begin{proof}
    Clearly $F[v_j]=0$. Assume $v_j$ is stable. Choose
    \[
        \eta = \min\left\{c,\frac{\la_0}{4K},\frac{\delta}{4K}\right\}>0,
    \]
    where the constant $c>0$ satisfies \eqref{eq:c}, so that for any $\|u\|_*<c$, $\|R[u]\|_*\leq K\|u\|_*^2$.

    Note that $0<\delta<\la_0/2$, hence $\eta=\min\{c,\delta/(4K)\}>0$. Then there exists $0<\ep<\eta$ such that when $\|u_0\|_*\leq \ep$, the solution $u(s)$ exists for all $s\ge 0$ and satisfies $\|u(s)\|_* < \eta$.
    
    Note that $\la_0=\max\{\la_1-\la_j,-2\la_j\}>0$ is an eigenvalue of $L_j$. Choose $\phi\in X$ to be a unit real eigenvector for $\la_0$. Take the initial perturbation $u_0=\ep \phi$. Then the stability condition holds. Thus $\|x(s)\|_* \leq \|u(s)\|_* < \eta$ for all $s\ge0$.

    Let $a(s) = \|x(s)\|_*$, $b(s) = \|y(s)\|_*$, and define the set
    \[
        S = \left\{ s \ge 0 : b(\tau) \le \eta^{-1} a^2(\tau),\quad \forall \tau \in [0,s] \right\}.
    \]
    Since $y(0)=0$, $x(0)=\ep\phi$, and by continuity, $0\in S$, so $s_0 := \sup S > 0$.

    \noindent \textbf{Step 1. $s_0$ is finite.}
    By the definition of $S$, on $[0,s_0)$ we have $b \le \eta^{-1} a^2$. Combined with the stability assumption $a < \eta$, we have $b \le a$. Hence for any $s\in [0,s_0)$,
    \begin{align*}
        \df{a^2}{s}&=2\re\ip{x,\df{x}{s}}_*\\
        &=2\re\ip{x,L_jx}_*+2\re\ip{x,P_XR[x+y]}_*\\
        &\geq 2\re\ip{x,L_jx}_*-2\left|\ip{x,P_XR[x+y]}_*\right|\\
        &\ge 2\la_0 a^2 - 2a K(a^2+b^2)\\
        &\ge 2\la_0 a^2 - 4 K a^3,
    \end{align*}
    where the second inequality follows from \eqref{eq:xLjx} and \eqref{eq:Rxy}, noting that $\|x+y\|_*<\eta\leq c$.
    
    Since $a<\eta \le \la_0/(4K)$, we have $4 K a^3 \leq \la_0 a^2$, so
    \[
        \df{a^2}{s} \geq \la_0 a^2,\quad \forall s\in [0,s_0),
    \]
    thus
    \[
        a(s) \ge \ep \exp\left(\frac{\la_0 s}{2}\right), \quad \forall s\in[0,s_0).
    \]
    By continuity,
    \[
        s_0\leq\frac{2}{\la_0}\ln\left(\frac{\eta}{\ep}\right)
    \]
    In particular, $a$ is positive on $[0,s_0]$. This shows that $a$ is at most non-differentiable at $s=0$, so
    \begin{equation}\label{eq:dasq}
        \df{a}{s}\geq \frac{\la_0}{2}a,\quad a.e.\,s\in[0,s_0).
    \end{equation}   
    \noindent \textbf{Step 2. Derive a contradiction.}
    Since $s_0$ is finite, by the definition of $s_0$ and continuity, we must have $b(s_0) = \eta^{-1} a^2(s_0)$. Similarly, for all $s\in[0,s_0)$,
    \begin{align*}
        \df{b^2}{s}&=2\re\ip{y,\df{y}{s}}_*\\
        &=2\re\ip{y,L_jy}_*+2\re\ip{y,(I-P_X)R[x+y]}_*\\
        &\leq 2\re\ip{y,L_jy}_*+2\left|\ip{y,P_XR[x+y]}_*\right|\\
        &\leq 2(\la_0-\delta) b^2 + 2b K(a^2+b^2)\\
        &\leq 2(\la_0-\delta) b^2 + 4 K a^2b,
    \end{align*}
    where the second inequality follows from \eqref{eq:yLjy} and \eqref{eq:Rxy}. Since $y\in C^1([0,+\infty),H)$, for $0\leq s_1<s_2$,
    \begin{align*}
        |b(s_1)-b(s_2)|&=\big|\|y(s_1)\|_*-\|y(s_2)\|_*\big|\\
        &\leq \|y(s_1)-y(s_2)\|_*\\
        &\leq\max_{s\in[s_1,s_2]}\left\|\frac{dy}{ds}\right\|_{*}\cdot |s_1-s_2|,
    \end{align*}
    so $b$ is locally Lipschitz, hence differentiable almost everywhere. At points of differentiability where $b(s)>0$, we have
    \begin{equation}\label{eq:dbsq}
        \df{b}{s}\leq (\la_0-\delta) b + 2 K a^2,\quad a.e.\, s\in [0,s_0),
    \end{equation}
    and at points where $b(s)=0$, $b'(s)=0$, so the inequality also holds.
    
    Let
    \[
        r(s)=\frac{b(s)}{a^2(s)},\quad \forall s\in[0,s_0],
    \]
    Since $b$ is locally Lipschitz and $a\geq \ep$, $r$ is locally Lipschitz. Combining \eqref{eq:dasq} and \eqref{eq:dbsq}, we obtain
    \begin{align*}
        \df{r}{s}&=\df{b}{s}\cdot \frac{1}{a^2}-\frac{2b}{a^3}\cdot \df{a}{s}\\
        &\leq \frac{(\la_0-\delta) b + 2 K a^2}{a^2}-\frac{2b}{a^3}\cdot \frac{\la_0}{2}a\\
        &=-\delta r+2K,\quad a.e.\, s\in[0,s_0).
    \end{align*}
    Noting that $r(0)=0$, by Gronwall's inequality,
    \[
        r(s)\leq \frac{2K}{\delta}(1-e^{-\delta s}),\quad\forall s\in[0,s_0).
    \]
    By continuity,
    \[
        \eta^{-1}=r(s_0)\leq \frac{2K}{\delta}(1-e^{-\delta s_0})<\frac{2K}{\delta},
    \]
    which contradicts the choice $\eta=\min\{c,\delta/(4K)\}$.

    This shows that the initial assumption ($v_j$ stable) is false; hence $v_j$ is an unstable equilibrium point of the nonlinear evolution equation \eqref{eq:gadv}. The instability of $w_j$ follows similarly.
\end{proof}

For sufficiently small $\ep>0$, define the perturbed operators
\begin{align*}
    L_0^\ep v&:=L_0v+\ep \Delta v=Df(\ga)v-\frac{dv}{dt}+\ep\frac{d^2v}{dt^2},\quad \forall v\in H^2_\per,\\
    L_0^{\ep,*} w&:=L_0^*w+\ep \Delta w=Df(\ga)w+\frac{dw}{dt}+\ep\frac{d^2w}{dt^2},\quad \forall w\in H^2_\per,
\end{align*}
The corresponding perturbed GAD systems are then
\begin{equation}\label{eq:Fep}
    \begin{aligned}
        \pf{v(s,t)}{s}&=F^{\ep}[v] := L_0^{\ep} v(s,t) - \bip{v, L_0^{\ep} v}\, v(s,t),\\
        \pf{w(s,t)}{s}&=F^{\ep,*}[w] := L_0^{\ep,*} w(s,t) - \bip{w, L_0^{\ep,*} w}\, w(s,t).
    \end{aligned}
\end{equation}
Direct calculation gives
\[
    \df{}{s}\bip{v,v}=2\bip{v, L_0^{\ep} v} (1 - \bip{v,v}),
\]
hence the perturbed GAD system still preserves unit norm for real vectors $v\in \m R^n$.

This section will prove the asymptotic stability of the perturbed GAD system, i.e., Theorem \ref{thm:vepjnonstab}. To this end, we make the following assumptions:

\textbf{A6}. \textit{The eigenvalues of the operators $L_0^\ep,L_0^{\ep,*}$ can be listed as $\{\la_{j,k}^\ep:1\leq j\le n,\,k\in \m Z\}$ and are analytic functions of $\ep$, i.e.,
\[
    \la_{j,k}^\ep=\la_{j,k}+\ep \la_{j,k}^{(1)}+\ep^2 \la_{j,k}^{(2)}+\cdots.
\]
The corresponding eigenvectors can also be listed as $\{v_{j,k}^\ep:1\leq j\le n,\,k\in \m Z\}$ and $\{w_{j,k}^\ep:1\leq j\le n,\,k\in \m Z\}$, respectively, and have analytic expansions
\begin{align*}
    v_{j,k}^\ep&=v_{j,k}+\ep v_{j,k}^{(1)}+\ep^2 v_{j,k}^{(2)}+\cdots,\\
    w_{j,k}^\ep&=w_{j,k}+\ep w_{j,k}^{(1)}+\ep^2 w_{j,k}^{(2)}+\cdots.
\end{align*}}

\textbf{A7}. \textit{There exists $\ep_0>0$ such that for all $0<\ep<\ep_0$, $L_0^\ep$ retains the following properties of $L_0$: the eigenvalues of $L_0^\ep$ have algebraic multiplicity $1$, the eigenvectors form a Riesz basis for $H$, and additionally, $\la_{1,0}^\ep, v_{1,0}^\ep, w_{1,0}^\ep$ remain real-valued.}

If $L_0^\ep v^\ep=\la^\ep v^\ep$, substituting into the definition (similarly for $L_0^{\ep,*}$) yields
\[
    L_0 v + \ep (L_0 v^{(1)} + v'') + O(\ep^2) = \lambda v + \ep (\lambda v^{(1)} + \lambda^{(1)} v) + O(\ep^2),
\]
Equating zeroth-order terms gives $L_0 v=\la v$. Equating first-order coefficients gives
\[
    (L_0 - \lambda) v^{(1)} = - v'' + \lambda^{(1)} v. 
\]
Acting on both sides with $\conj{w}\in \ker(L_0^*-\conj{\la})$ and taking $\bip{w,v}=\ip{v,\conj{w}}=1$, we obtain $\la^{(1)}=\bip{w,v''}$. Applying this to $\la_{j,k}$ gives
\[
    \la_{j,k}^{(1)}=\bip{w_{j,k},v''_{j,k}},
\]
which is the first-order expansion coefficient.

Denote the subscript $j=(j,0)$ for simplicity, and normalize such that $\bip{v_1,w_j}=1$. Integrating by parts yields
\[
    \bip{w_{1,k},v''_{1,k}}=\bip{w_1,v''_1}-\om_k^2-2i\cdot \om_k\bip{w_1,v_1'},
\]
where the frequency $\om_k=2\pi k/T$. Note that $w_1,v_1$ are real vectors, hence
\[
    \re(\la_{1,k}^{(1)}-\la_1^{(1)})=-\om_k^2\leq -\frac{4\pi^2}{T}<0,\quad \forall k\neq 1. 
\]
This implies
\[
    \re(\la_{1,k}^\ep-\la_1^\ep)=-\ep \om_k^2+O(\ep^2)\leq -\frac{4\pi^2\ep}T+O(\ep^2).
\]
Using similar methods, for all $j\geq 2$, we can compute
\begin{align*}
     \re(\la_{j,k}^\ep-\la_1^\ep)&=(\la_j-\la_1)-\ep(\om_k^2+\al_j \om_k+\beta_j)+O(\ep^2)\\
     &\leq -\la_1+O(\ep),
\end{align*}
where $\al_j=2\im\bip{w_j,v_j'},\;\beta_j=\bip{w_1,v_1''}-\re\bip{w_j,v_j''}$.

Therefore, to facilitate the stability analysis, we make the following regularity assumption:

\textbf{A8}. \textit{There exists $\ep_0>0$ such that for all $0<\ep<\ep_0$ and $1\leq j\leq n, k\in \m Z$,
\[\re(\la_{j,k}^\ep-\la_1^\ep)\leq -\delta(\ep)<0.\]}

The linearization of the nonlinear functional $F^\ep$ at $v_1^\ep$ is the operator $L^\ep:H^2_\per\subset H\to H$ given by
\begin{equation}\label{eq:Lep}
    \begin{aligned}
    L^\ep v :&= \left.\frac{d}{dh}\right|_{h=0}F[v_1^\ep+h v]\\
    &= L_0^\ep v - \la_1^\ep v - \bip{v_1^\ep, L_0^\ep v}v_1^\ep - \la_1^\ep\bip{v, v_1^\ep}v_1^\ep.
    \end{aligned}
\end{equation}
Similarly, the linearization of $F^{\ep,*}$ at $w_1^\ep$ is $L^{\ep,*}:H^2_\per\subset H\to H$ given by
\begin{equation}\label{eq:Lepstar}
    \begin{aligned}
    L^{\ep,*} w :&= \left.\frac{d}{dh}\right|_{h=0}F^*[w_1^\ep+h w]\\
    &= L^{\ep,*} w - \la_1^\ep w - \bip{w_1^\ep, L_0^{\ep,*} w}w_1^\ep - \la_1^\ep\bip{w, w_1^\ep}w_1^\ep.
    \end{aligned}
\end{equation}
Fix $\ep$ and denote the eigenvalues of $L^\ep$ as $\ka_\al: \al\in \m Z$ with corresponding eigenvectors $e_\al$. To introduce an equivalent inner product on $H$ such that $\{e_\al\}$ forms an orthonormal basis:
\[
    \ip{e_\al,e_\beta}_*=\delta_{\al,\beta}.
\]
Then for any $u=\sum c_\al e_\al\in H^2_\per$,
\begin{equation}\label{eq:Lepu}
    \re \ip{u,L^\ep u}_*=\sum \re \conj{\ka_\al} |c_\al|^2\leq -\sigma(\ep)\cdot \|u\|_*^2,
\end{equation}
and for $u\in H^3_\per$,
\begin{equation}\label{eq:Lepptu}
   \re \ip{\p_t u,L^\ep \p_t u}_*\leq -\sigma(\ep)\cdot \|\p_t u\|_*^2.
\end{equation}
Consider the commutator
\[
    [\partial_t, L^\ep]u = \partial_t(L^\ep u) - L^\ep(\partial_t u), \quad\forall u\in H^3_\per(0,T),
\]
Since $L_0^\ep = Df(\ga) - \partial_t + \ep \partial_t^2$, we have $[\p_t, L_0^\ep]=[\p_t, Df(\ga)]=\p_t Df(\ga)$.

Using $L^\ep = L_0^\ep - \lambda_1^\ep - A_1 - \la_1^\ep A_2$, where
\[
    A_1 u = \bip{v_1^\ep, L_0^\ep u}\, v_1^\ep,\qquad A_2 u = \bip{u, v_1^\ep}\, v_1^\ep,
\]
we obtain
\[
    [\partial_t, L^\ep] = \p_t Df(\ga) - [\partial_t, A_1] - [\partial_t, A_2].
\]
For any $u\in H^3_\per$, noting that $L_0^\ep u\in H^1_\per$, integration by parts gives
\begin{align*}
    [\p_t, A_1]u&=\p_t (A_1u)-A_1(\p_t u)\\
    &=\bip{v_1^\ep,L_0^\ep u}\cdot \p_t v_1^\ep-\bip{v_1^\ep,L_0^\ep(\p_t u)}\cdot v_1^\ep\\
    &=\bip{v_1^\ep,L_0^\ep u}\cdot \p_t v_1^\ep-\bip{v_1^\ep,\p_t (L_0^\ep u)} v_1^\ep+\bip{v_1^\ep, (\p_t Df(\ga))u}v_1^\ep\\
    &=\bip{Df(\ga)^\top v_1^\ep+\p_t v_1^\ep+\ep \p_t^2 v_1^\ep, u}\cdot \p_t v_1^\ep\\
    &\quad +\bip{Df(\ga)^\top \p_t v_1^\ep+\p_t^2 v_1^\ep+\ep \p_t^3 v_1^\ep, u}\cdot v_1^\ep\\
    &\quad +\bip{(\p_t Df(\ga)^\top) v_1^\ep,u}\cdot v_1^\ep,
\end{align*}
hence $\|[\p_t, A_1]u\|_H^2\leq K\|u\|_H^2$.

Similarly,
\begin{align*}
    [\p_t, A_2]u&=\p_t(A_2u)-A_2(\p_t u)\\
    &=\bip{u,v_1^\ep}\cdot \p_t v_1^\ep-\bip{\p_t u,v_1^\ep} v_1^\ep\\
    &=\bip{u,v_1^\ep}\cdot \p_t v_1^\ep+\bip{u,\p_t v_1^\ep} v_1^\ep,
\end{align*}
also yielding $\|[\p_t, A_2]u\|_H^2\leq K\|u\|_H^2$. Therefore $[\partial_t, L^\ep]$ is a bounded linear operator on $H$. Without loss of generality, denote its bound in the $\|\cdot\|_*$ norm also by $K$.

Assume $v\in H^3_\per$ is a solution of the perturbed GAD system \eqref{eq:Fep}. Let $u(s,t)=v(s,t)-v_1^\ep(t)$. Then $u\in H^3_\per$ satisfies the equation
\[
    \df{u}{s}=L^\ep u+R[u],
\]
where the remainder term $R[u]$ is
\begin{align*}
    R[u]&= F^\ep[v_1^\ep+u] - L^\ep u \\
    &=-\bip{v_1^\ep,L_0^\ep u}u-\la_1^\ep\bip{u,v_1^\ep}u\\
    &\quad -\bip{u,L_0^\ep u}v_1^\ep-\bip{u,L_0^\ep u}u.
\end{align*}
Let $C=\max\left\{\|v_1^\ep\|_H, \|\p_t v_1^\ep\|_H, \|\p_t^2 v_1^\ep\|_H,\max_{t\in [0,T]}\|Df(\ga(t))\|\right\}$. Using integration by parts, we estimate
\begin{align*}
    |\bip{v_1^\ep,L_0^\ep u}|&=|\bip{Df(\ga)^\top v_1^\ep +\p_t v_1^\ep +\ep \p_t^2 v_1^\ep,u}|\leq \|h\|_H\cdot \|u\|_H\\
    |\bip{u,L_0^\ep u}|&=|\bip{u,Df(\ga)u}-\ep\bip{\p_t u,\p_t u}|\leq C\|u\|_H^2+\ep\|\p_t u\|_H^2
\end{align*}
where $h=Df(\ga)^\top v_1^\ep +\p_t v_1^\ep +\ep \p_t^2 v_1^\ep$.

Thus, there exists a constant $K$ such that for all $\|u\|_H\leq 1$,
\[
    \|R[u]\|_H\leq K(\|u\|_H^2+\|\p_t u\|_H^2).
\]
Using a completely analogous method, for all $\max\{\|u\|_H,\|\p_t u\|_H\}\leq 1$,
\[
    \|\p_t R[u]\|_H\leq K(\|u\|_H^2+\|\p_t u\|_H^2).
\]
Since $\|\cdot\|_*$ is equivalent to $\|\cdot\|_H$, by \eqref{eq:c}, there exists a constant $c>0$ such that when $\max\{\|u\|_*,\|\p_t u\|_*\}\leq c$,
\begin{equation}\label{eq:Ru}
    \max\{\|\p_t R[u]\|_*, \|R[u]\|_*\} \leq K (\|u\|_*^2+\|\p_t u\|_*^2).
\end{equation}
At this point, we can state the stability theorem for the perturbed GAD system.
\begin{theorem}\label{thm:vepjnonstab}
    Under assumptions \textbf{A6}--\textbf{A8}, for any $0<\ep<\ep_0$, the real directions $(v_1^\ep,\p_t v_1^\ep)$ and $(w_1^\ep,\p_t w_1^\ep)$ are asymptotically stable equilibrium points of the perturbed GAD system $(F^\ep[v],\p_t F^\ep[v])$ and the adjoint perturbed GAD system $(F^{\ep,*}[w],\p_t F^{\ep,*}[w])$, respectively, where both $v_1^\ep$ and $w_1^\ep$ are unit vectors.
    
    More precisely, $(v_1^\ep,\p_t v_1^\ep)$ and $(w_1^\ep,\p_t w_1^\ep)$ are asymptotically stable equilibrium points on the infinite-dimensional manifold
    \[
        M:=\left\{(x,y)\in H^3_\per\times H^2_\per: y=\p_t x\right\}
    \]
    for the evolution equations
    \[
        \frac{\p x}{\p s} = F^\ep[x],\quad \frac{\p y}{\p s} = \p_t F^\ep[x]
    \]
    and
    \[
        \frac{\p x}{\p s} = F^{\ep,*}[x],\quad \frac{\p y}{\p s} = \p_t F^{\ep,*}[x],
    \]
    respectively.
\end{theorem}
\begin{proof}
Define the energy functional
\[
E(s) = \frac{1}{2}\left(\|u(s)\|_*^2 + \al \|\p_t u(s)\|_*^2\right),
\]
where $\al > 0$ is to be determined.

\noindent\textbf{Step 1. Energy estimate}. Differentiating with respect to $s$,
\[
    \frac{dE}{ds} = \re \ip{u, \p_s u}_* + \al \re \ip{\p_t u, \p_s \p_t u}_*.
\]
Substituting $\p_s u = L^{\ep} u + R[u]$ and using $\p_s \p_t u = L^{\ep} \p_t u + [\p_t, L^{\ep}] u + \p_t R[u]$, we obtain
\[
\begin{aligned}
    \frac{dE}{ds} &= \re \ip{u, L^{\ep} u}_* + \re \ip{u, R[u]}_* + \al \re \ip{\p_t u, L^{\ep} \p_t u}_* \\
    &\quad+ \al \re \ip{\p_t u, [\p_t, L^{\ep}] u}_* + \al \re \ip{\p_t u, \p_t R[u]}_*.
\end{aligned}
\]
Using the boundedness of $[\p_t, L^\ep]$ and the estimates \eqref{eq:Lepu}, \eqref{eq:Lepptu}, \eqref{eq:Ru}, we get
\begin{align*}
    \frac{dE}{ds}&\leq -\sigma E+\al K \|\p_t u\|_*\cdot \|u\|_*+K\|u\|_*(\|u\|_*^2+\|\p_t u\|_*^2)\\
    &\quad +\al K\|\p_t u\|_*(\|u\|_*^2+\|\p_t u\|_*^2)\\
    &\leq -\sigma E+\frac{\al\sigma}{4}\|\p_t u\|_*^2 +\frac{\al K^2}{\sigma}\|u\|_*^2+K\eta (\|u\|_*^2+\|\p_t u\|_*^2)\\
    &\quad +\al K\eta (\|u\|_*^2+\|\p_t u\|_*^2)\\
    &=-\sigma E+\left(\frac{\al K^2}{\sigma}+(1+\al)K\eta\right)\|u\|_*^2+\left(\frac{\al\sigma}{4}+(1+\al)K\eta\right)\|\p_t u\|_*^2.
\end{align*}
It can be seen that when
\[
    0<\al\leq \min\left\{1,\frac{\sigma^2}{16K^2}\right\},\quad \max\left\{\|u_*\|,\|\p_t u\|_*\right\}< \eta\leq \min\left\{c,\frac{\al\sigma}{8(1+\al)K}\right\}
\]
we have
\[
    \frac{dE}{ds}\leq -\frac{\sigma}{2}E.
\]
\noindent\textbf{Step 2. A priori estimate and continuation}. We now prove that $\|u\|_*$ and $\|\p_t u\|_*$ remain below $\eta$ during the evolution, so that the above estimate holds for all $s\ge0$. Choose the initial perturbation sufficiently small such that
\[
    \|u(0)\|_* \le \frac{1}{2} \sqrt{\al}\,\eta,\quad \|\p_t u(0)\|_* \le \frac{1}{2} \eta.
\]
Then $E(0) = \|u(0)\|_*^2 + \al \|\p_t u(0)\|_*^2 \le \frac{1}{2} \al \eta^2$. Define
\[
    s_0 = \sup\{ s \ge 0 : \max\{\|u(\tau)\|_*,\|\p_t u(\tau)\|_*\} < \eta,\quad \forall \tau\in [0,s]\}.
\]
By the initial conditions and continuity of solutions, $s_0 > 0$. On the interval $[0,s_0)$, the conditions hold, so the derived estimate is valid, giving
\[
    \frac{dE}{ds} \le -\frac{\sigma}{2} E, \quad \forall  s\in[0,s_0).
\]
Thus $E(s) \le E(0) \exp(-\frac{\sigma}{2}s)$ holds on $[0,s_0)$. Consequently,
\[
\|u(s)\|_*^2 \le E(s) \le E(0) \le \frac{1}{2} \al \eta^2,\quad
\al \|\p_t u(s)\|_*^2 \le E(s) \le E(0)\leq \frac{1}{2} \al \eta^2,
\]
which yields
\[
\|u(s)\|_* \le \sqrt{\frac{1}{2} \al}\,\eta < \eta,\quad
\|\p_t u(s)\|_* \le \sqrt{\frac{1}{2}}\,\eta < \eta.
\]
Therefore, for all $s\in [0,s_0)$, $\max\{\|u(s)\|_*,\|\p_t u(s)\|_*\}\leq \wt{\eta}<\eta$.

If $s_0$ is finite, by continuity of the solution at $s=s_0$, we also have 
\[
    \max\{\|u(s_0)\|_* ,\|\p_t u(s_0)\|_*\} \le \wt{\eta} < \eta,
\] 
contradicting the definition of $s_0$. Hence $s_0 = +\infty$. Thus the energy estimate holds for all $s\ge0$.

\noindent\textbf{Step 3. Asymptotic stability}. Then,
\[
    \|u(s)\|_*^2+\al\|\p_t u(s)\|_*^2=E(s) \le E(0) \exp\left(-\frac{\sigma}{2}s\right), \quad \forall s\geq 0.
\]
By the equivalence of the inner products, $\|u(s)\|_H$ and $\|\p_t u(s)\|_H$ also decay exponentially to zero, proving that $(v_1^{\ep},\p_t v_1^\ep)$ is asymptotically stable. The asymptotic stability of $(w_1^{\ep},\p_t w_1^\ep)$ follows similarly.
\end{proof}
Now assume that a $T$-periodic orbit with one unstable direction has been located using the algorithm in table \ref{tab:AdjointPeriod}:
\[
    \ga(\tau)=\frac{a_0}{2}+\sum_{k=1}^{N}\left(a_k\cos(2\pi k\tau)+b_k\sin(2\pi k\tau)\right),
\]
where $\tau=t/T$ is the normalized time. Let the unknown adjoint unstable direction be
\[
    w(s,\tau)=\frac{\al_0(s)}{2}+\sum_{k=1}^{N}\left(\al_k(s)\cos(2\pi k\tau)+\beta_k(s)\sin(2\pi k\tau)\right).
\]
Fix $m\geq 1$ and a perturbation coefficient $\ep>0$. Denote $\tau_j=j/m, w_j=w(s,\tau_j)$ and
\begin{align*}
    Q_j&=Df\left(x(\tau_j)\right)^\top w_j-\frac{T}{m}\left(\sum_{i=0}^{m-1}w_i^\top Df\left(x(\tau_i)\right)^\top w_i\right)w_j,\\
    S &= \frac{T}{2}\sum_{k=1}^{N} (2\pi k)^2 (\al_k^2 + \beta_k^2),
\end{align*}
then the coefficients satisfy the perturbed GAD equations
\begin{equation}\label{eq:PeriodStbmfd}
    \begin{aligned}
        \frac{d\al_0}{ds} &= \frac{2}{m}\sum_{j=0}^{m-1} Q_j \;+\; \ep S \al_0,\\
        \frac{d\al_k}{ds} &= \frac{2}{m}\sum_{j=0}^{m-1} Q_j\cos(2\pi k\tau_j) \;+\; \frac{2\pi k}{T}\,\beta_k
        -\ep (2\pi k)^2 \al_k + \ep S \al_k,\\
        \frac{d\beta_k}{ds} &= \frac{2}{m}\sum_{j=0}^{m-1} Q_j\sin(2\pi k\tau_j) \;-\; \frac{2\pi k}{T}\,\al_k
        -\ep (2\pi k)^2 \beta_k + \ep S \beta_k,
    \end{aligned}
\end{equation}
Based on this, the numerical algorithm used in this paper to compute the stable manifold of a periodic orbit with one unstable direction is:
\begin{table}[!htbp]
    \caption{\enspace Numerical algorithm for the stable manifold of a periodic orbit $\ga$ with one unstable direction.}
    \label{tab:PeriodStbmfd}
    \centering
    \begin{tabularx}{\textwidth}{>{\centering\arraybackslash}m{1.3cm}X}
        \toprule[1.5pt] 
        \textbf{S1} & Guess the initial shape coefficients $\al_k^0,\beta_k^0$ of the adjoint unstable direction of $\ga$.\\
        \midrule
        \textbf{S2} & Starting from the initial values $\al_k^0,\beta_k^0$, integrate along the vector field \eqref{eq:PeriodStbmfd} for a long time and observe convergence. If not, repeat step \textbf{S1}.\\
        \bottomrule[1.5pt]
    \end{tabularx}
\end{table}
\section{Numerical Experiments}
\label{sec:NE}
This section presents three numerical experiments: a two-machine system, a three-machine system containing only saddles, and a three-machine system with a periodic orbit. In all three systems, we uniformly take $p=50\,\text{Hz}$.

\subsection{Two-Machine System}
First, we give the parameters for $n=2$:
\begin{equation}\label{eq:G2}
    \begin{gathered}
        H=(6.5,6.5),\quad D=(0.1,0.1),\quad P=(0,0),\quad E=(1,1,1), \\
        G=\begin{pmatrix}
            0 & -0.02 & 0.02\\
            -0.02 & -0.05 & 0.03\\
            0.02 & 0.03 & 0
        \end{pmatrix},\quad 
        B=\begin{pmatrix}
            0 & 0.4996 & 0.9998\\
            0.4996 & 0 & 0.4991\\
            0.9998 & 0.4991 & 0
        \end{pmatrix}.
    \end{gathered}
\end{equation}
Applying the algorithm in table \ref{tab:1saddleGAD} to system \eqref{eq:del} with the above parameters, we obtain three 1-saddles near the asymptotically stable equilibrium point $(0,0)$ (due to the periodicity of the vector field, there are actually six 1-saddles; see Figure \ref{fig:G2_DOA}) and their associated unstable eigen-directions, and then compute the boundary of the region of attraction. The experimental results are shown in Figure \ref{fig:G2}.
\begin{figure}[!htbp]
    \centering
    \caption{\enspace Two-machine system.}
    \label{fig:G2}
    \begin{subfigure}[b]{0.49\textwidth}
        \centering
        \includegraphics[width=\textwidth]{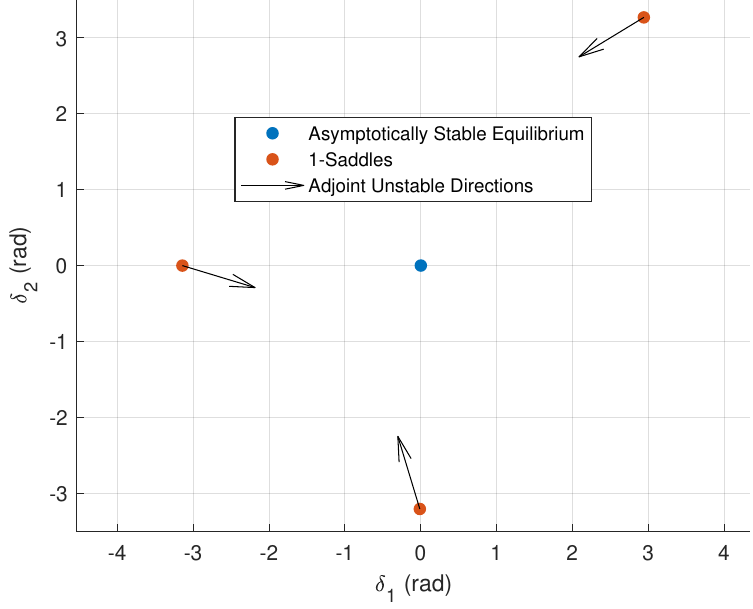}
        \caption{\enspace 1-saddles and their associated unstable eigen-directions.}
        \label{fig:G2_saddle}
    \end{subfigure}
    \begin{subfigure}[b]{0.49\textwidth}
        \centering
        \includegraphics[width=\textwidth]{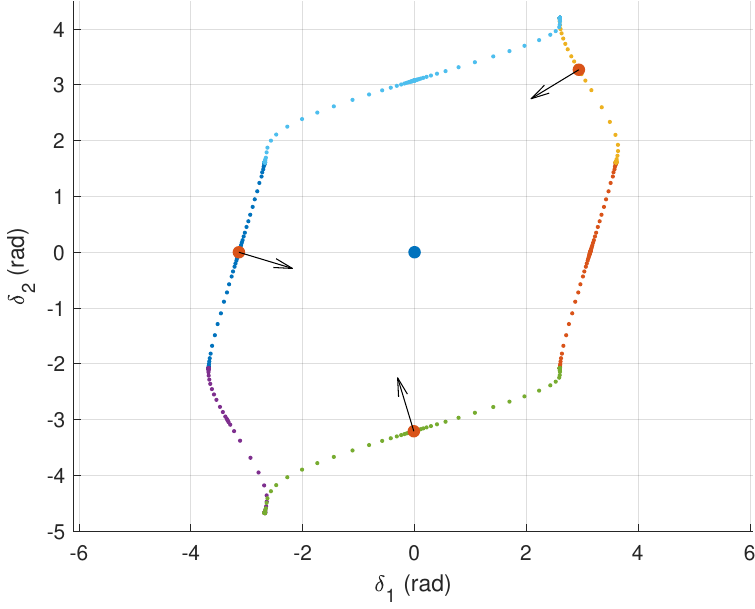}
        \caption{\enspace Boundary of the region of attraction.}
        \label{fig:G2_DOA}
    \end{subfigure}
\end{figure}

\subsection{Three-Machine System}
First, we give the parameters for $n=3$:
\begin{equation}\label{eq:G3}
    \begin{gathered}
        H=(6.5,6.5,6.5),\quad D=(0.1,0.1,0.1),\quad  P=(0,0,0),\\
        E=(1,1,1,1),\quad
        G=\begin{pmatrix}
            0 & 0 & 0 & 0\\
            0 & 0 & 0 & 0\\
            0 & 0 & 0 & 0\\
            0 & 0 & 0 & 0
        \end{pmatrix},\quad 
        B=\begin{pmatrix}
            1 & 1 & 1 & 1\\
            1 & 1 & 1 & 1\\
            1 & 1 & 1 & 1\\
            1 & 1 & 1 & 1
        \end{pmatrix}.
    \end{gathered}
\end{equation}
Applying the algorithm in table \ref{tab:1saddleGAD} to system \eqref{eq:del} with the above parameters, we obtain four 1-saddles near the asymptotically stable equilibrium point $(0,0,0)$ and their associated unstable eigen-directions, and then compute the boundary of the region of attraction. The experimental results are shown in Figure \ref{fig:G3}.
\begin{figure}[!htbp]
    \centering
    \caption{\enspace Three-machine system.}
    \label{fig:G3}
    \begin{subfigure}[b]{0.49\textwidth}
        \centering
        \includegraphics[width=\textwidth]{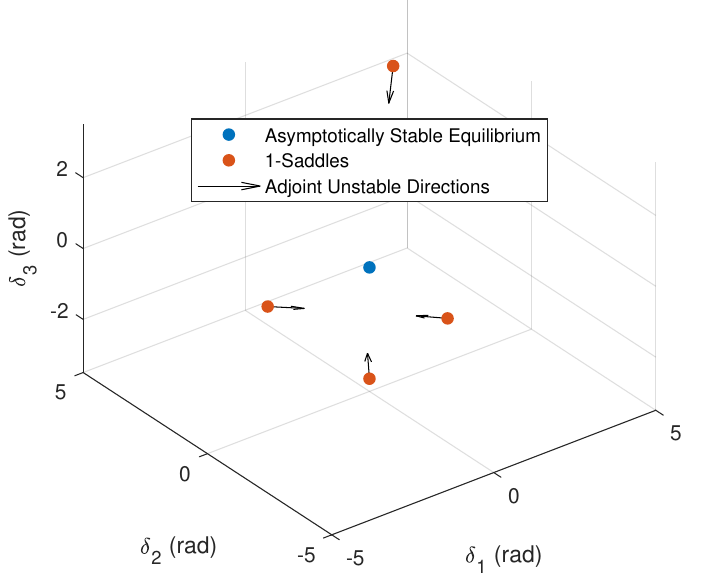}
        \caption{\enspace 1-saddles and their associated unstable eigen-directions.}
        \label{fig:G3_saddle}
    \end{subfigure}
    \begin{subfigure}[b]{0.49\textwidth}
        \centering
        \includegraphics[width=\textwidth]{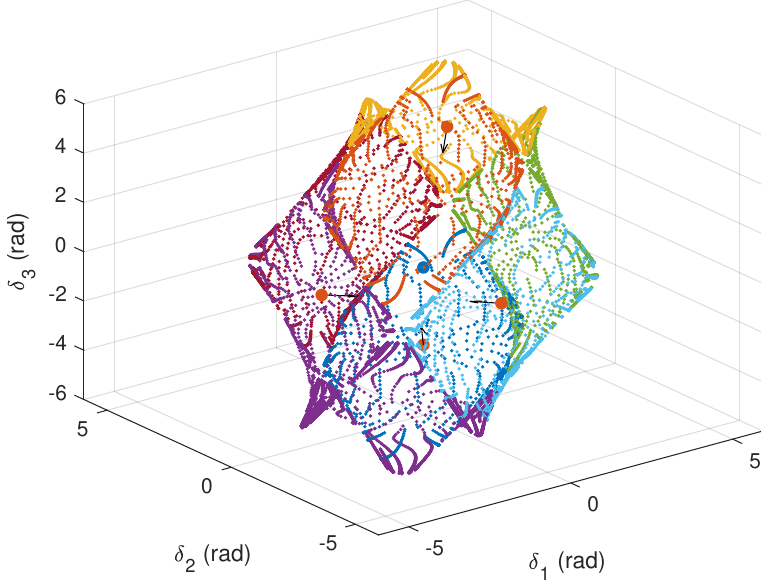}
        \caption{\enspace Boundary of the region of attraction.}
    \end{subfigure}
\end{figure}

\subsection{Three-Machine System with a Periodic Orbit}
First, we give the parameters for $n=3$:
\begin{equation}\label{eq:G3_Period}
    \begin{aligned}
        H&=(6.5,6.5,6.175),\\
        D&=(0.1,0.1,0.095),\\
        P&=(0,0,0,0),\\
        E&=(1.071, 1.0599, 1.0692, 1.0528),\\
        G&=\begin{pmatrix}
            0.1516 & 0.1698 & 0.0470 & 0.0846\\
            0.1698 & 0.3574 & 0.0847 & 0.1517\\
            0.0470 & 0.0847 & 0.2147 & 0.2745\\
            0.0846 & 0.1517 & 0.2745 & 0.5309
        \end{pmatrix},\\
        B&=\begin{pmatrix}
            -1.2583 & 0.9712 & 0.0625 & 0.1077\\
            0.9712  & -1.4828 & 0.1078 & 0.1775\\
            0.0652 & 0.1078 & -1.2944 & 0.9044\\
            0.1077 & 0.1775 & 0.9044 & -1.6044
        \end{pmatrix}.
    \end{aligned}
\end{equation}
Applying the algorithm in table \ref{tab:1saddleGAD} to system \eqref{eq:del} with the above parameters, we obtain one 1-saddle near the asymptotically stable equilibrium point $(0,0,0)$ and its associated unstable eigen-direction, and then compute the stable manifold of this saddle. The experimental results are shown in Figure \ref{fig:G3_general_saddle}.
\begin{figure}[!htbp]
    \centering
    \caption{\enspace 1-saddle and its stable manifold for the three-machine system with a periodic orbit.}
    \label{fig:G3_general_saddle}
    \begin{subfigure}[b]{0.49\textwidth}
        \centering
        \includegraphics[width=\textwidth]{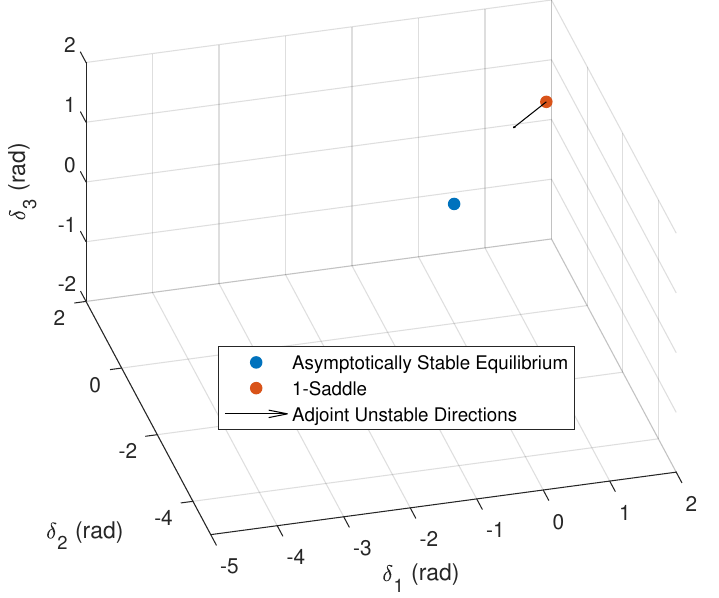}
        \caption{\enspace 1-saddle and its associated unstable eigen-direction.}
    \end{subfigure}
     \begin{subfigure}[b]{0.49\textwidth}
        \centering
        \includegraphics[width=\textwidth]{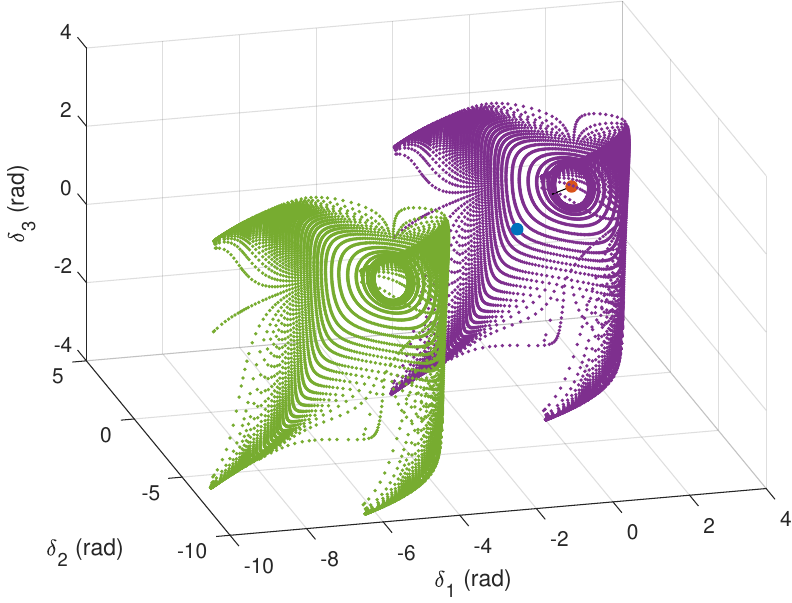}
        \caption{\enspace Stable manifold of the 1-saddle.}
    \end{subfigure}
\end{figure}

It can be seen that the stable manifold of the 1-saddle forms only two faces of the boundary of the region of attraction; it is conjectured that each of the remaining four faces contains a periodic orbit of index one\footnote{Due to the periodicity of the vector field \eqref{eq:del} in each direction with period $2\pi$, the term \textit{periodic orbit} here actually refers to $x(\tau)=2\pi \tau d + y(\tau)$, where $\tau=t/T$ is normalized time, $d\in \{-1,0,1\}^n$ is a spatial periodic direction of $x$ in \eqref{eq:del}, and $y(0)=y(1)$. Thus $y$ is a genuine periodic orbit. The period, location, and stable manifold of $y$ can be computed numerically, from which the corresponding geometric quantities of $x$ are naturally obtained.}. Using the algorithm in table \ref{tab:AdjointPeriod} to locate the periodic orbit and the algorithm in table \ref{tab:PeriodStbmfd} to compute its associated unstable eigen-direction, the experimental results are shown in Figure \ref{fig:G3_general_period}.
\begin{figure}[!htbp]
    \centering
    \caption{\enspace Periodic orbit and its associated unstable eigen-direction for the three-machine system with a periodic orbit.}
    \label{fig:G3_general_period}
    \begin{subfigure}[b]{0.49\textwidth}
        \centering
        \includegraphics[width=\textwidth]{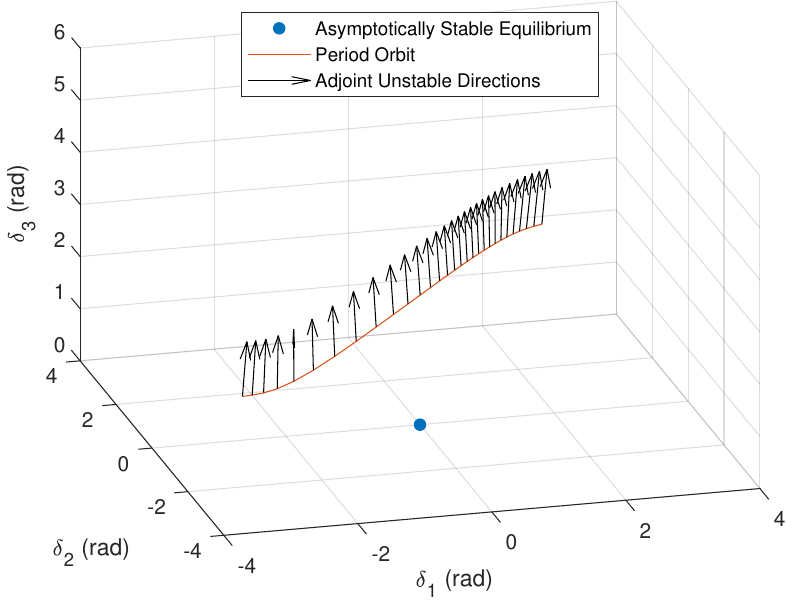}
        \caption{\enspace Periodic orbit located on the upper side of the region of attraction.}
        \label{fig:G3_general_period_up}
    \end{subfigure}
    \begin{subfigure}[b]{0.49\textwidth}
        \centering
        \includegraphics[width=\textwidth]{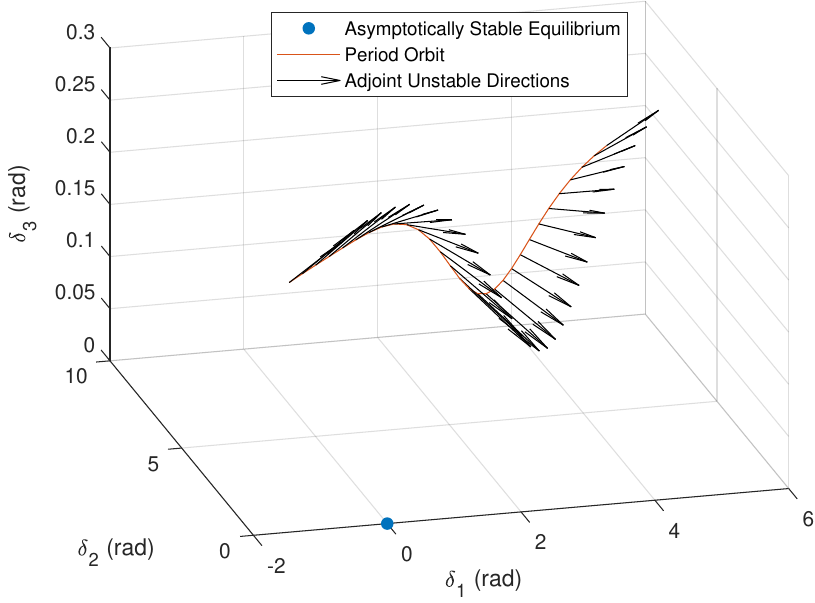}
        \caption{\enspace Periodic orbit located on the left side of the region of attraction.}
        \label{fig:G3_general_period_left}
    \end{subfigure}
\end{figure}
Using the unstable eigen-direction of the periodic orbit, its stable manifold is computed, and the experimental results are shown in Figure \ref{fig:G3_general_period_stabmfd}.
\begin{figure}[!htbp]
    \centering
    \caption{\enspace Stable manifolds of the periodic orbits for the three-machine system with a periodic orbit.}
    \label{fig:G3_general_period_stabmfd}
    \begin{subfigure}[b]{0.49\textwidth}
        \centering
        \includegraphics[width=\textwidth]{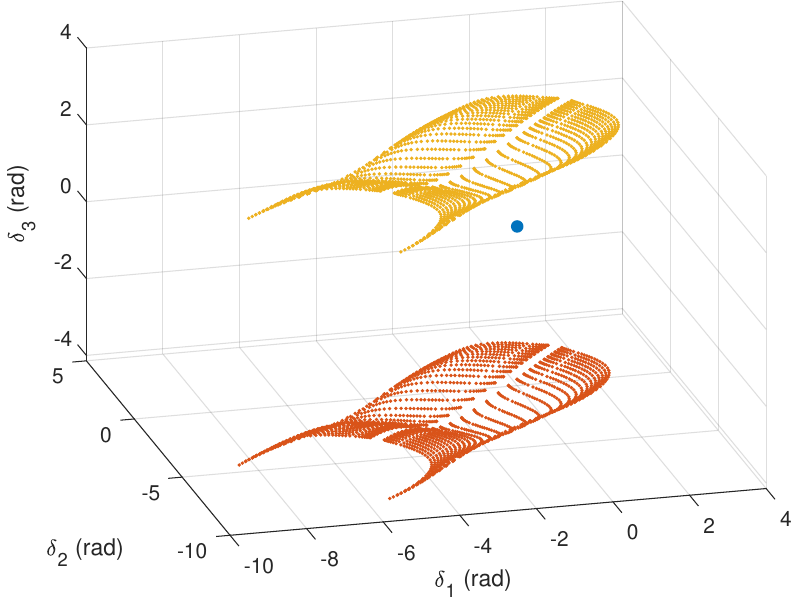}
        \caption{\enspace Stable manifolds of the periodic orbits on the upper and lower faces.}
    \end{subfigure}
    \begin{subfigure}[b]{0.49\textwidth}
        \centering
        \includegraphics[width=\textwidth]{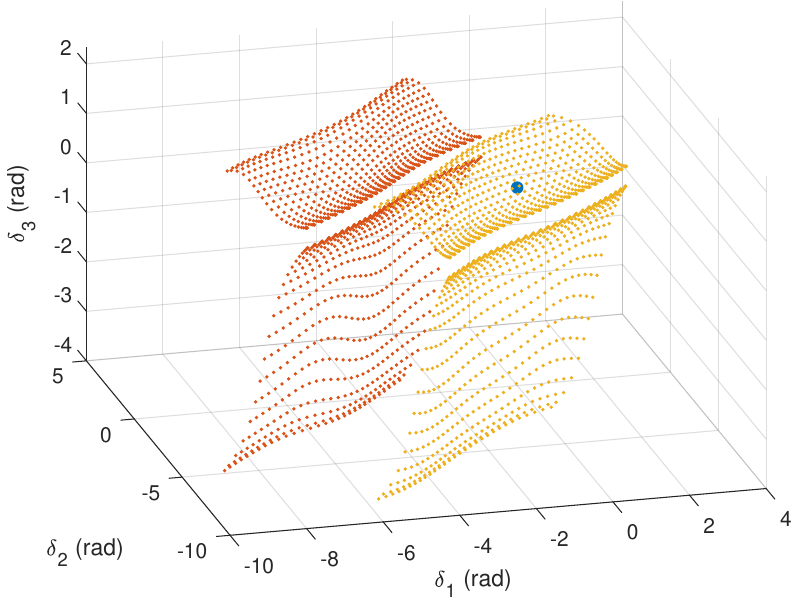}
        \caption{\enspace Stable manifolds of the periodic orbits on the left and right faces.}
    \end{subfigure}
\end{figure}
The stable manifolds of the six faces together constitute the boundary of the region of attraction, as shown in Figure \ref{fig:G3_general_DOA}.
\begin{figure}[!htbp]
    \centering
    \caption{\enspace Boundary of the region of attraction for the three-machine system with a periodic orbit.}
    \label{fig:G3_general_DOA}
        \centering
        \includegraphics[width=0.5\textwidth]{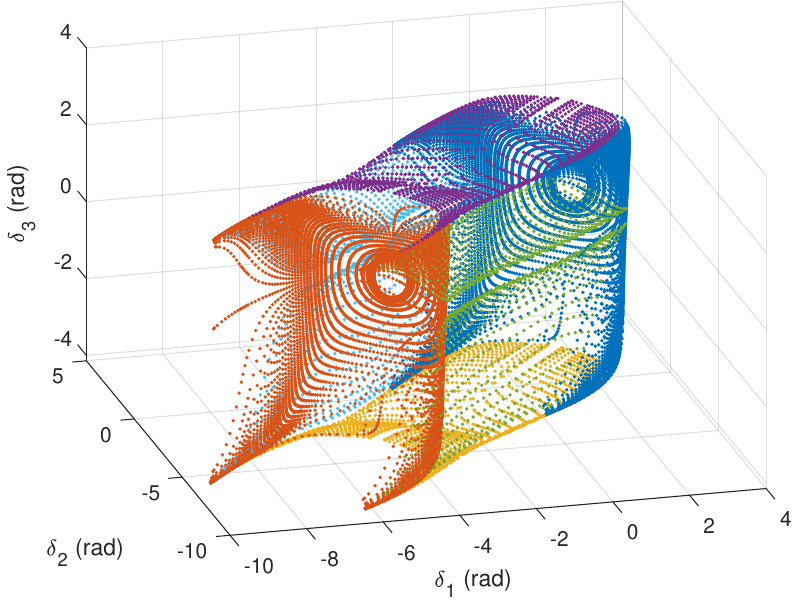}
\end{figure}

When applying the algorithm in table \ref{tab:PeriodStbmfd} to obtain Figure \ref{fig:G3_general_period_up}, we specifically set the perturbation coefficient $\ep=0.07$, the number of sampling points on the periodic orbit $m=30$, and the truncation order of the Fourier series $N=5$. In theory, the adjoint unstable eigen-direction $w(t)$ is orthogonal to $\ga'(t)$ at every point. table \ref{tab:wga_up} gives the cosines of the angles between $w(\tau_j)$ and $\ga'(\tau_j)$, where $1\leq j\leq 30$.
\begin{table}[!htbp]
    \caption{\enspace Cosines of the angles between the adjoint unstable eigen-direction and the tangent vector of the periodic orbit for Figure \ref{fig:G3_general_period_up}.}
    \label{tab:wga_up}
    \centering
    \small
    \setlength{\tabcolsep}{7pt}
    \begin{tabular}{
        S[table-format=-1.6]
        S[table-format=-1.6]
        S[table-format=-1.6]
        S[table-format=-1.6]
        S[table-format=-1.6]
        S[table-format=-1.6]
    }
        \toprule
        -0.105752 & -0.141213 & -0.137720 & -0.096649 & -0.046490 & -0.012324 \\
        0.003312 &  0.008898 &  0.010685 &  0.009790 &  0.005925 &  0.001326 \\
        -0.000596 &  0.000124 & -0.000573 & -0.005352 & -0.011280 & -0.013309 \\
        -0.011165 & -0.009956 & -0.012732 & -0.014305 & -0.004559 &  0.021497 \\
        0.055455 &  0.077118 &  0.071835 &  0.042727 & -0.000561 & -0.052543 \\
        \bottomrule
    \end{tabular}
\end{table}

For the reversed vector field $-f$ and the time-reversed periodic orbit $\ga(T-t)$, using an algorithm similar to that in table \ref{tab:PeriodStbmfd} we compute the unstable eigen-direction $v(T-t)$ of $\ga(T-t)$, which is precisely the stable eigen-direction of the forward-time periodic orbit $\ga(t)$. In theory $w(t)$ is orthogonal to $v(t)$ at every point. table \ref{tab:wv_up} gives the cosines of the angles between $w(\tau_j)$ and $v(\tau_j)$.
\begin{table}[!htbp]
    \caption{\enspace Cosines of the angles between the adjoint unstable eigen-direction and the stable eigen-direction for Figure \ref{fig:G3_general_period_up}.}
    \label{tab:wv_up}
    \centering
    \small
\setlength{\tabcolsep}{7pt}
\begin{tabular}{
    S[table-format=-1.6]
    S[table-format=-1.6]
    S[table-format=-1.6]
    S[table-format=-1.6]
    S[table-format=-1.6]
    S[table-format=-1.6]}
        \toprule
        0.000753 & 0.000084 & -0.000471 & -0.000815 & -0.000871 & -0.000705 \\
        -0.000472 & -0.000281 & -0.000161 & -0.000095 & -0.000055 & -0.000032 \\
        -0.000025 & -0.000042 & -0.000079 & -0.000127 & -0.000176 & -0.000238 \\
        -0.000335 & -0.000460 & -0.000545 & -0.000492 & -0.000246 & 0.000206 \\
        0.000947 & 0.001976 & 0.002812 & 0.002918 & 0.002348 & 0.001527 \\
        \bottomrule
    \end{tabular}
\end{table}

Similarly, when applying the algorithm in table \ref{tab:PeriodStbmfd} to obtain Figure \ref{fig:G3_general_period_left}, we specifically set the perturbation coefficient $\ep=0.005$, the number of sampling points on the periodic orbit $m=30$, and the truncation order of the Fourier series $N=5$. Tables \ref{tab:wga_left} and \ref{tab:wv_left} give the cosines of the angles between $w(\tau_j)$ and $\ga'(\tau_j)$, and between $w(\tau_j)$ and $v(\tau_j)$, respectively.
\begin{table}[!htbp]
    \centering
    \small
    \setlength{\tabcolsep}{7pt}
    \caption{\enspace Cosines of the angles between the adjoint unstable eigen-direction and the tangent vector of the periodic orbit for Figure \ref{fig:G3_general_period_left}.}
    \label{tab:wga_left}
    \begin{tabular}{
        S[table-format=-1.6]
        S[table-format=-1.6]
        S[table-format=-1.6]
        S[table-format=-1.6]
        S[table-format=-1.6]
        S[table-format=-1.6]
    }
        \toprule
        0.000629 & 0.000600 & 0.000622 & 0.000653 & 0.000619 & 0.000501 \\
        0.000330 & 0.000127 & -0.000134 & -0.000452 & -0.000745 & -0.000900 \\
        -0.000890 & -0.000795 & -0.000710 & -0.000649 & -0.000571 & -0.000460 \\
        -0.000361 & -0.000305 & -0.000257 & -0.000155 & 0.000012 & 0.000187 \\
        0.000316 & 0.000406 & 0.000499 & 0.000604 & 0.000676 & 0.000677 \\
        \bottomrule
    \end{tabular}
\end{table}

\begin{table}[!htbp]
    \centering
    \small
    \setlength{\tabcolsep}{7pt}
    \caption{\enspace Cosines of the angles between the adjoint unstable eigen-direction and the stable eigen-direction for Figure \ref{fig:G3_general_period_left}.}
    \label{tab:wv_left}
    \begin{tabular}{
        S[table-format=-1.6]
        S[table-format=-1.6]
        S[table-format=-1.6]
        S[table-format=-1.6]
        S[table-format=-1.6]
        S[table-format=-1.6]
    }
        \toprule
        -0.000005 & -0.000021 & -0.000005 &  0.000002 & -0.000033 & -0.000090 \\
        -0.000118 & -0.000106 & -0.000089 & -0.000096 & -0.000099 & -0.000065 \\
        -0.000008 &  0.000022 &  0.000002 & -0.000039 & -0.000061 & -0.000061 \\
        -0.000068 & -0.000097 & -0.000121 & -0.000110 & -0.000069 & -0.000036 \\
        -0.000034 & -0.000037 & -0.000013 &  0.000034 &  0.000059 &  0.000037 \\
        \bottomrule
    \end{tabular}
\end{table}

\subsection{Level Set Characterization}
In the engineering practice of power systems, it is necessary to automatically and quickly determine whether the system starting from a specific state (i.e., some $\delta\in \m R^n$) can return to the stable state (i.e., the asymptotically stable equilibrium point $0$). At this point, it is neither feasible to rely on human operators to visually judge whether the state point is enclosed by the computed point cloud of the boundary of the region of attraction, nor is there often sufficient time for time-domain simulation. Therefore, this section uses neural network methods to solve the level sets of the boundaries of the region of attraction for the three examples in the previous section. After precomputing the level set neural network, one can quickly determine whether the system can return to the steady state by directly evaluating the level function value at the state point.

All models adopt a fully connected neural network architecture $f_{\theta}:\mathbb{R}^n\to [0,+\infty)$, which follows the design principles of deep feedforward networks that have been successfully applied to high-dimensional function approximation tasks \cite{Bengio2013, LeCun1998}. The network parameters $\theta = \{W_i, b_i:1\leq i\leq 4\}$ include weight matrices $W_i$ and bias vectors $b_i$ for 3 hidden layers and 1 output layer. The network is defined as
\[
    f_{\theta}(x) = \|x\|^2 \cdot \bigl( \operatorname{Softplus}(\mathrm{MLP}_{\theta}(x)) + 10^{-4} \bigr),
\]
where $\operatorname{Softplus}(z)=\ln(1+e^z)$ is a smooth approximation of the ReLU activation, which has been shown to accelerate training in deep networks \cite{Hinton2012}, and $\mathrm{MLP}_{\theta}$ is a multilayer perceptron consisting of 3 hidden layers (128 neurons each, activation function $\tanh$) and 1 linear output layer. Clearly $f_{\theta}$ is nonnegative and $f_{\theta}(0)=0$.

Given the point set $\left\{x_i:1\leq i\leq N\right\}\subset\m{R}^n$ on the boundary of the region of attraction computed in the previous section as the training point set, the loss function is
\[
    \ma L(\theta)=\ma L_{\text{fit}}+\ma L_{\text{grad}} +\ma L_{\text{ext}},
\]
where the fitting term
\[
   \ma L_{\text{fit}}  = w_{\text{fit}} \cdot \frac{1}{N}\sum_{i=1}^{N}\left(f_{\theta}\left(x_i\right)-1\right)^2
\]
imposes a penalty at the known boundary sample points $x_i$; minimizing this loss drives $f_\theta(x_i)$ towards $1$, thereby constraining the unit level set $\{x: f_\theta(x)=1\}$ to lie near these sample points.

The gradient penalty term is
\[
    \ma L_{\text{grad}} = w_{\text{grad}} \cdot \frac{1}{N}\sum_{i=1}^{N}\left(\tau - \left\|D_x f_{\theta}\left(x_i\right)\right\|\right)_+^2,
\]
where $(\cdot)_+=\max(0,\cdot)$ and $\tau\geq 1$. Intuitively, if the function has a larger gradient near $x_i$, the distance required along the radial path from the origin $0$ to $x_i$ for the function value to grow from $0$ to $1$ will be smaller, thereby constraining the level set $f_\theta=1$ near the training points and preventing it from expanding outward excessively.

The exterior point constraint term is
\[
\mathcal{L}_{\text{ext}} = w_{\text{ext}} \cdot \frac{1}{M}\sum_{j=1}^{M}\bigl(t - f_{\theta}(x_j^{\text{ext}})\bigr)_+^2
\]
where $t>1$, and the exterior point set $\{x_j^{\text{ext}}\}$ is generated by radial extension: for each training point $x_i$, extend radially along $u_i=x_i/\|x_i\|$ by a distance $d = \max(r\|x_i\|,\,0.002)$ to obtain $x^{\text{ext}} = x_i + d u_i$, with $r$ a preset radial extension ratio. If $f_\theta(x_j^{\text{ext}}) < t$, the loss term drives the function value upward. Since the fitting term makes $f_\theta(x_i) \approx 1$ and the exterior point is located a distance $d$ radially outward from the corresponding training point and is required to satisfy $f_\theta(x_j^{\text{ext}})\ge t>1$, the function must increase from approximately $1$ to at least $t$ over the radial distance $d$, i.e., the radial derivative must be at least approximately $(t-1)/d$. This makes it difficult for the level set $f_\theta=1$ to expand outward and stay close to the training points.

To prevent overfitting from the finite set of training points, techniques such as early stopping and gradient penalization are used. Dropout regularization \cite{Srivastava2014} has also proven effective in similar geometric learning tasks, though we rely here on the explicit $L_{\text{grad}}$ and $L_{\text{ext}}$ terms for stability.

All models use the Adam optimizer with a learning rate of $10^{-3}$, implemented with the PyTorch framework, and random seeds (Python, NumPy, PyTorch) are fixed to 42. Figures \ref{fig:fit_G2} to \ref{fig:fit_G3_Period} demonstrate the effective fitting results of this method for different systems. This success echoes the broader trend in which deep learning has dramatically improved performance in recognition and representation tasks \cite{ Hinton2012, Krizhevsky2012}.

For the two-machine system \eqref{eq:G2}, the level set $f=0.9$ is shown in Figure \ref{fig:fit_G2}. The loss function weights are $w_{\text{fit}}=300$, $w_{\text{origin}}=2$, $w_{\text{grad}}=5$, $w_{\text{ext}}=10$, the gradient threshold $\tau=1.0$, the target output value $t=1.2$, the radial extension ratios $r\in\{0.005,\,0.01\}$, and the number of training epochs is 9000.
\begin{figure}[!htbp]
    \centering
    \caption{\enspace Level set of the boundary of the region of attraction for the two-machine system.}
    \label{fig:fit_G2}
        \centering
        \includegraphics[width=0.5\textwidth]{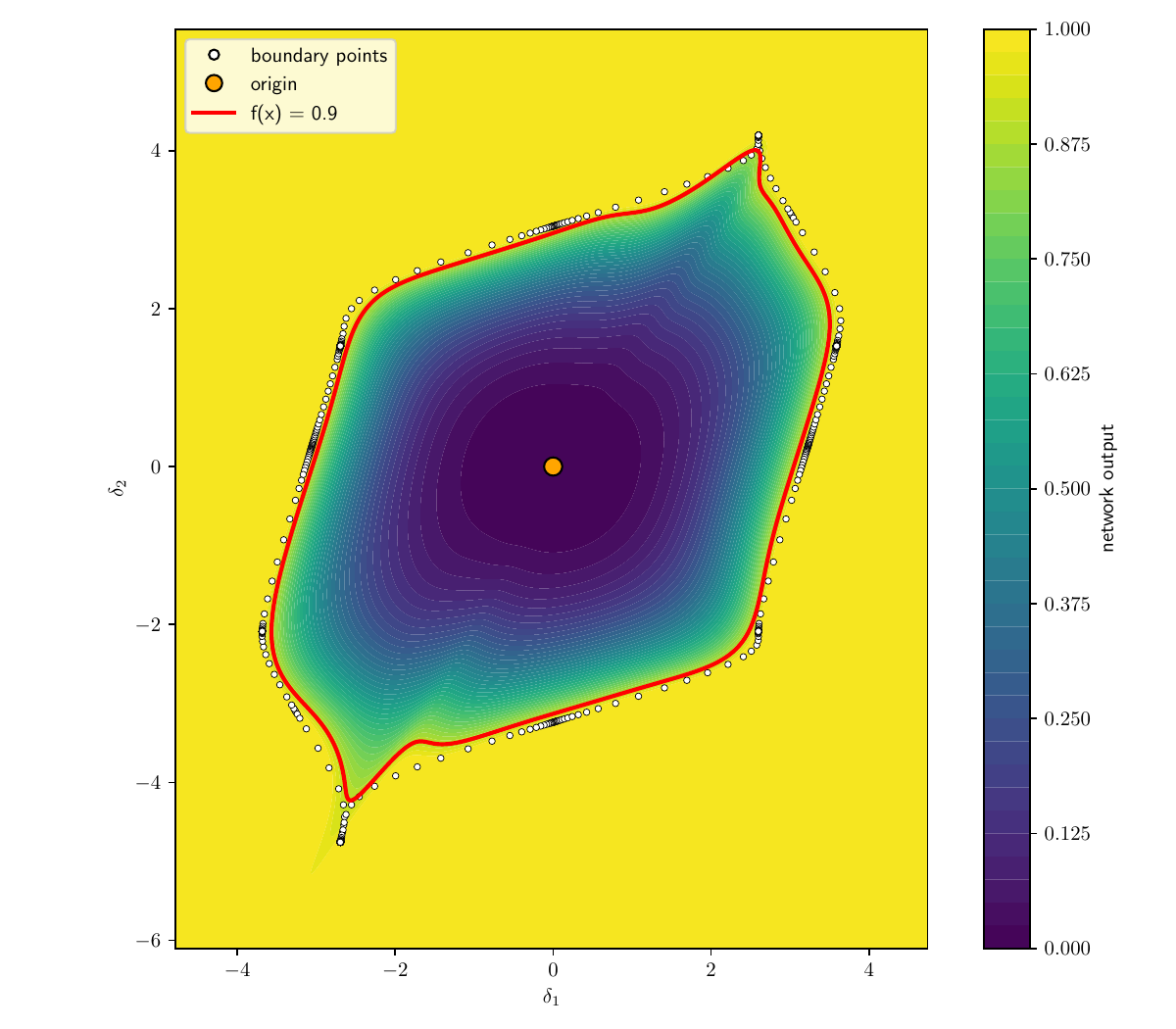}
\end{figure}

For the three-machine system \eqref{eq:G3}, the level set $f=0.9$ is shown in Figure \ref{fig:fit_G3}. The loss function weights are $w_{\text{fit}}=500$, $w_{\text{origin}}=1$, $w_{\text{grad}}=0$, $w_{\text{ext}}=20$, the target output value $t=1.1$, the radial extension ratios $r\in\{0.02,\,0.05\}$, and the number of training epochs is 6500.
\begin{figure}[!htbp]
    \centering
    \caption{\enspace Level set of the boundary of the region of attraction for the three-machine system.}
    \label{fig:fit_G3}
        \centering
        \includegraphics[width=0.5\textwidth]{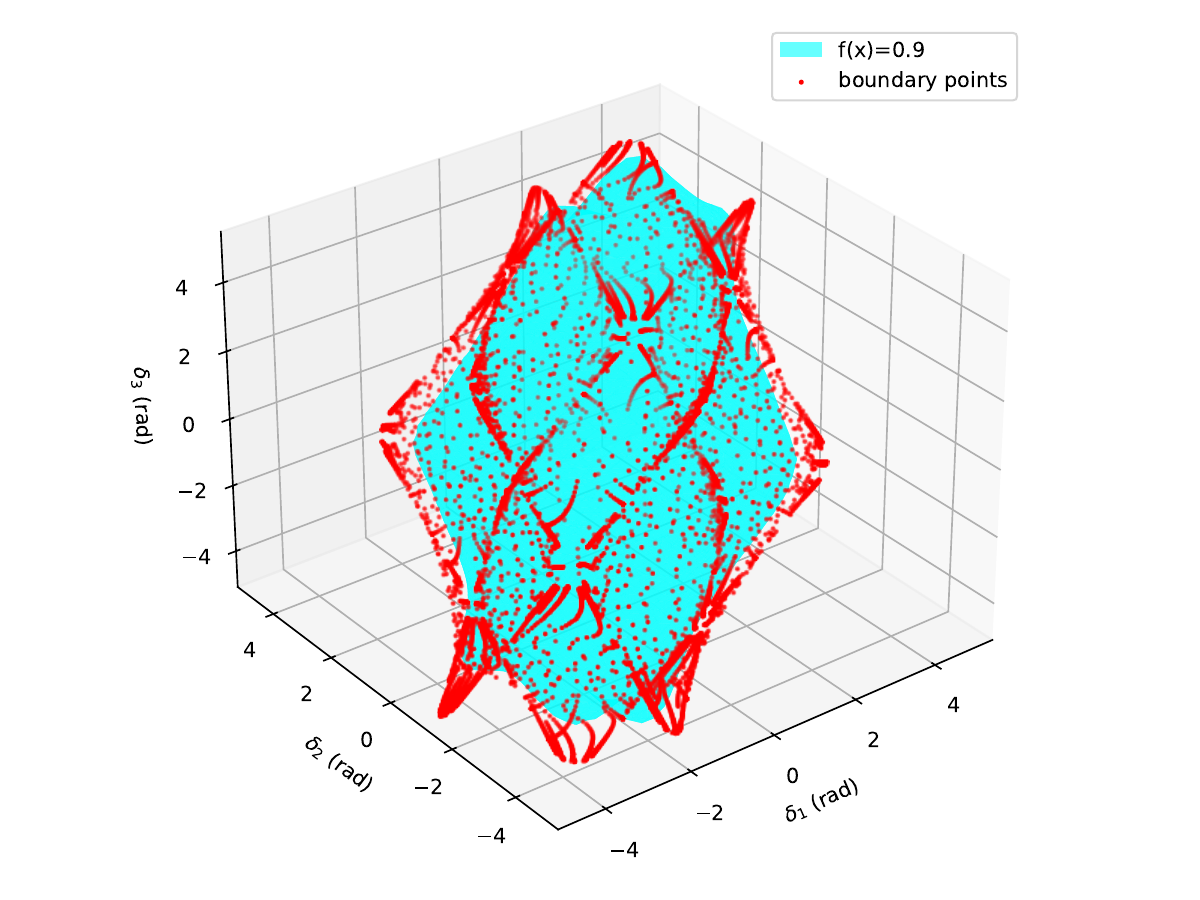}
\end{figure}

For the three-machine system with a periodic orbit \eqref{eq:G3_Period}, the level set $f=0.97$ is shown in Figure \ref{fig:fit_G3_Period}. The loss function weights are $w_{\text{fit}}=500$, $w_{\text{origin}}=1$, $w_{\text{grad}}=0$, $w_{\text{ext}}=20$, the target output value $t=1.1$, the radial extension ratios $r\in\{0.02,\,0.05\}$, and the number of training epochs is 6500.
\begin{figure}[!htbp]
    \centering
    \caption{\enspace Level set of the boundary of the region of attraction for the three-machine system with a periodic orbit.}
    \label{fig:fit_G3_Period}
        \centering
        \includegraphics[width=0.5\textwidth]{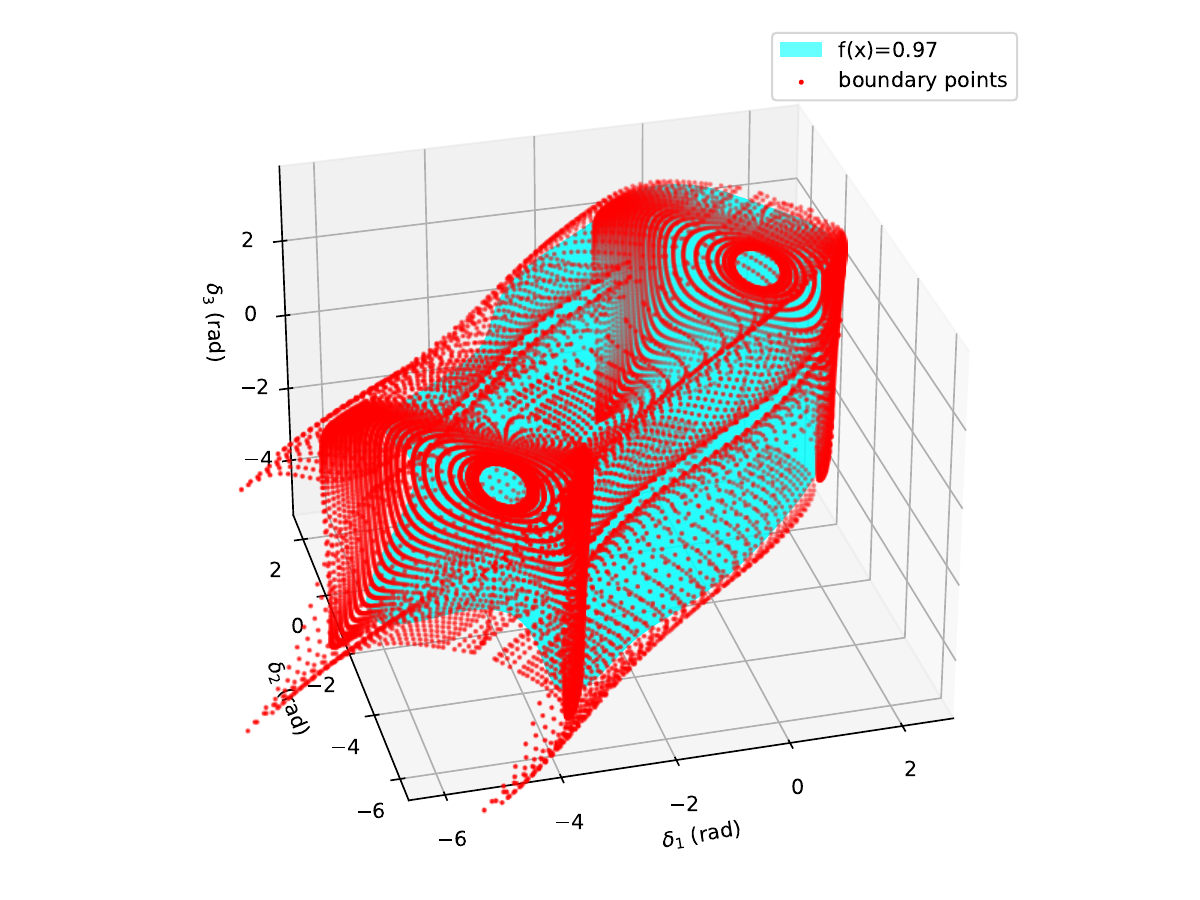}
\end{figure}

From Figures \ref{fig:fit_G2} to \ref{fig:fit_G3_Period}, it can be seen that the trained (conservative) level set can well reflect the boundary of the region of attraction, and can also learn its shape well at non-smooth points on the boundary. The $\|x\|^2$ factor forces the function to be zero at the origin, and combined with the universal approximation capability of the MLP, it can approximate closed surfaces of complex shapes without requiring a priori assumptions about the boundary morphology. However, limitations are also evident: for example, the loss weights $w_{\text{fit}}, w_{\text{grad}}, w_{\text{ext}}$, the gradient threshold $\tau$, the exterior point target value $t$, and the radial extension ratio $r$ need to be manually tuned for each specific problem, lacking automatic selection criteria; improper settings may lead to an overly tight or loose boundary, or training divergence. Moreover, the network structure does not explicitly enforce radial monotonicity; although numerical experiments have not produced multiple level sets or holes, there is no theoretical guarantee of the radial monotonicity of $f_\theta$. These limitations await further research and improvement.
\section{Conclusion}
\label{sec:conclusion}
This paper makes progress in both theoretical analysis and algorithm design for the numerical computation of the domain of attraction boundary of synchronous generator systems. Theoretically, by introducing a topological homeomorphism assumption based on classical results, it is proved that the domain of attraction boundary can be constituted by the closure of the union of the stable manifolds of critical elements with index one (saddle points or periodic orbits), providing a foundation for numerical computation. Algorithmically, the gentlest ascent dynamics method is employed for 1-saddle points, transforming them into asymptotically stable equilibrium points of an augmented system. For periodic orbits with index one, an adjoint operator-based localization algorithm is used, and the GAD method is extended to infinite-dimensional settings for computing stable manifolds of periodic orbits. By introducing artificial viscosity, a stability theory for the perturbed system is established, providing theoretical guarantees for computing stable manifolds of periodic orbits. Numerical experiments on a two-machine system, a three-machine system containing only saddle points, and a three-machine system containing periodic orbits validate the effectiveness and geometric accuracy of the proposed algorithms, offering a new numerical tool for transient stability analysis of power systems.
\section{Acknowledgments}
This study is supported by Energy Development Research institute of China Southern Power Grid (Grant No. EDRI-XT-ZLYJ-2025-102) and the National Natural Science Foundation of China (Grant No. 92470119).
\bibliographystyle{unsrtnat}
\bibliography{ref}
\end{document}